\documentclass{article}

\usepackage{amsthm, amsmath, amsfonts, mathrsfs, amssymb, bbm}

\usepackage[colorlinks=true, citecolor=violet, linkcolor=blue, filecolor=magenta, urlcolor=cyan]{hyperref}
\usepackage{geometry}

\makeatletter
\newcommand*\rel@kern[1]{\kern#1\dimexpr\macc@kerna}
\newcommand*\widebar[1]{%
	\begingroup
	\def\mathaccent##1##2{%
		\rel@kern{0.8}%
		\overline{\rel@kern{-0.8}\macc@nucleus\rel@kern{0.2}}%
		\rel@kern{-0.2}%
	}%
	\macc@depth\@ne
	\let\math@bgroup\@empty \let\math@egroup\macc@set@skewchar
	\mathsurround\z@ \frozen@everymath{\mathgroup\macc@group\relax}%
	\macc@set@skewchar\relax
	\let\mathaccentV\macc@nested@a
	\macc@nested@a\relax111{#1}%
	\endgroup
}
\makeatother

\makeatletter
\def\keywords{\xdef\@thefnmark{}\@footnotetext}
\makeatother

\usepackage{cite}
\usepackage{enumitem}
\usepackage{cleveref}
\usepackage{authblk}
\usepackage{xcolor}

\newtheorem{thm}{Theorem}[section]
\newtheorem*{thm*}{Theorem}
\newtheorem{prop}[thm]{Proposition}
\newtheorem{coro}[thm]{Corollary}

\newtheorem{lem}[thm]{Lemma}
\theoremstyle{definition}
\newtheorem*{assum*}{Assumption}
\newtheorem{defn}[thm]{Definition}

\theoremstyle{remark}
\newtheorem{rmk}[thm]{Remark}

\makeatletter
\def\cref@thmoptarg[#1]#2#3#4{%
    \ifhmode\unskip\unskip\par\fi%
    \normalfont%
    \trivlist%
    \let\thmheadnl\relax%
    \let\thm@swap\@gobble%
    \thm@notefont{\fontseries\mddefault\upshape}%
    \thm@headpunct{.}
    \thm@headsep 5\p@ plus\p@ minus\p@\relax%
    \thm@space@setup%
    #2
    \@topsep \thm@preskip               
    \@topsepadd \thm@postskip           
    \def\@tempa{#3}\ifx\@empty\@tempa%
      \def\@tempa{\@oparg{\@begintheorem{#4}{}}[]}%
    \else%
      \refstepcounter[#1]{#3}
      \@namedef{cref@#3@alias}{#1}
      \def\@tempa{\@oparg{\@begintheorem{#4}{\csname the#3\endcsname}}[]}%
    \fi%
    \@tempa}%
\makeatother

\DeclareMathAlphabet\matheuvm{U}{zeur}{m}{n}
\newcommand{\di}{\matheuvm{D}}
\newcommand{\diff}{\operatorname{d}}

\newcommand{\citeinfo}[2]{\cite[#1]{#2}}

\title{Absolute continuity of generalized Wasserstein barycenters of finitely many measures}
\author{Jianyu Ma \quad
	}
\affil{Institut de Mathématiques de Toulouse}
\date{\today}

\begin{document}
\maketitle


\begin{abstract}
    Consider a complete Riemannian manifold $(M, \matheuvm{g})$
    and optimal transport problems on it
    with cost functions of the form $c(x,y) = h(d_{\matheuvm{g}}(x,y))$.
	We study the absolute continuity of the corresponding generalized Wasserstein barycenters
    of finitely many marginal measures.
    For general strictly convex profiles $h$ lacking $\mathcal{C}^2$-smoothness,
    such as $h(d)= d^p / p$ with $1 < p < 2$ that defines the $p$-Wasserstein space, 
    the singularity at $d=0$ prevents
    the barycenter from inheriting absolute continuity from a single marginal measure
    as the quadratic case.
    To overcome this singularity,
    recent Euclidean results \cite{brizzi2025p, brizzi2026h} necessitate the absolute continuity of all marginals.
    Building upon the approximation framework toward absolute continuity in \cite{ma2025absolute},
    we extend the Euclidean advancements to the manifold setting.
    Stripping away the implicit reliance on flat translational symmetry
    and local coordinate calculations of their Euclidean proofs,
    our work handles the singularity in a geometrically transparent way,
    revealing the precise analytic condition on the cost profile that governs the necessary assumptions.
\end{abstract}


\renewcommand\contentsname{}
\vspace{-1cm}
\tableofcontents

\section{Introduction}

Barycenter is the notion of mean for probability measures on metric spaces.
Given a set of probability measures $\mu_1, \ldots, \mu_n$ on a metric space $(E, d)$
with positive weights $\lambda_i > 0$ satisfying $\sum_{i=1}^n \lambda_i = 1$,
their (Wasserstein) barycenter is typically defined as a minimizer of the functional
\[
	\nu \mapsto \sum \lambda_i \, W_2(\nu, \mu_i)^2,
\]
where $W_2$ denotes the $2$-Wasserstein distance of probability measures.
On Euclidean spaces $\mathbb{R}^m$, Agueh and Carlier \cite{agueh2011barycenters} established the fundamental theory for these barycenters,
showing their existence, uniqueness, and absolute continuity when at least one marginal is absolutely continuous.
These results were subsequently extended to Riemannian manifolds by Kim and Pass \cite{kim2017wasserstein}
for the compact case and by the author \cite{ma2025absolute} for the general non-compact case,
who exploited the specific geometric properties of the squared distance function $d^2/2$.

Recall that the Wasserstein distance is derived from the optimal transport problem of Monge and Kantorovich.
Given two probability measures $\mu$ and $\nu$ on $E$, one seeks a transport plan, i.e.,
a probability measure $\gamma$ on $E \times E$ with marginals $\mu$ and $\nu$,
that moves $\mu$ to $\nu$ while, in addition, minimizing the total cost $\int c(x,y) \diff \gamma(x,y)$.
When the cost is a power of the distance, this defines a metric on the space of probability measures.
The geometry of this Wasserstein space, and consequently the regularity of its barycenters, is intimately linked to the analytical properties of the underlying cost function.

In recent years, attention has turned to generalizing this concept beyond the quadratic cost.
As pointed out by Carlier during the author's defense \cite{ma2025thesis}, in the Euclidean setting, Brizzi, Friesecke, and Ried have extensively studied $p$-Wasserstein barycenters \cite{brizzi2025p} and, more generally, $h$-Wasserstein barycenters \cite{brizzi2026h},
where the squared distance is replaced by $|x-y|^p$ ($p > 1$)
or $h(x - y)$ with $h$ satisfying certain properties such as strict convexity and $\mathcal{C}^2$-smoothness.

Extending these results to general costs on Riemannian manifolds
$(M, \matheuvm{g})$ presents fundamental analytic and geometric obstructions.
As in the quadratic case \cite[Proposition 2.8]{ma2025absolute}, if $\mu_1$ is absolutely continuous,
then $\mathbb{P} := \sum_{i=1}^n \lambda_i \, \delta_{\mu_i}$ has a unique generalized Wasserstein barycenter $\mu_\mathbb{P}: = B_{\#} \gamma$, where $\gamma$ is a multi-marginal transport plan for the measures $\mu_1, \ldots, \mu_n$
with respect to the generalized barycenter cost function on $M^n$,
\[
	C(x_1, \ldots, x_n) := \inf_{y \in M} \sum_{i=1}^n \lambda_i\, h(d_{\matheuvm g}(y, x_i)),
\]
and $B$ is a measurable map sending a configuration $\matheuvm x = (x_1, \ldots, x_n) \in M^n$ to a barycenter $z$ of $\sum_{i=1}^n \lambda_i\, \delta_{x_i}$.
In the Euclidean case, the analytic singularity at the diagonal $d_{\matheuvm{g}}(z, x_i) = \| z - x_i \|= 0$ (where general costs fail to be smooth) is handled by reducing
the barycenter map $B: M^n \rightarrow M$ on these collision regions $z = x_i$ to marginal projections $(x_1, \ldots, x_n) \mapsto x_i$ \cite[Lemma 3.4]{brizzi2025p}. To ensure these singular sets carry no distorted mass, they assume that \emph{all} marginals are absolutely continuous. However, outside these singularities, establishing inverse-Lipschitz bounds \cite[Proposition 3.2]{brizzi2025p} for the barycenter map heavily relies on implicit differentiation that exploits flat translational symmetry \cite[(5.5) in Remark 5.1]{brizzi2026h}. This reliance on multi-dimensional coordinate bounds irreparably fails on curved Riemannian manifolds due to the twisting of Jacobi fields.

In this article, we present a new, geometrically transparent framework that encompasses these Euclidean advancements into the discrete approximation strategy developed in \cite{ma2025absolute}.
By approximating $\mu_i$ for $2 \le i \le n$ with discrete measures, we disintegrate the support of $\gamma$ into conditional slices where the discrete points $(x_2, \ldots, x_n)$ act as fixed spatial anchors. This reduces the problem to bounding the local differential of the inverse transport map $F: z \mapsto x_1$ on these slices. Crucially, this entirely frees the proof from heavy coordinate calculations
and implicit flat-space requirements,
allowing us to handle it efficiently and formulate the natural additional assumptions it demands.

Our proof is thus separated into three steps:
reduce the absolute continuity to the discrete case via approximation, establish local Lipschitz control of $F$, and handle singular sets by decomposing $\operatorname{supp}(\gamma)$.
This framework improves the fundamental understanding of both Euclidean and Riemannian cases,
and is highly adaptable for future investigations.
Specifically, it yields two distinct resolutions to the collision singularity, which directly govern the measure-theoretic assumptions for $\mu_1, \ldots, \mu_n$:
\begin{enumerate}
    \item \emph{Analytical Smoothing:} If the cost profile $h$ is locally flat at the origin, satisfying $h''(0) := \lim_{t \downarrow 0} (h'(t) - h^{\prime}(0))/t = \lim_{t \downarrow 0} h''(t)$ (e.g., $h(d) = d^p/p$ for $p \ge 2$),
        the geometric singularity is neutralized thanks to
        the classical approximation of 
        $\operatorname{Hess}_{z} d_{\matheuvm{g}}(z, x_i)$ when $z \rightarrow x_i$.
        The Hessian extends continuously across the collision, preserving Lipschitz continuity directly. This powerful regularity allows us to require \emph{only a single absolutely continuous marginal $\mu_1$}.
	\item \emph{Geometric Exclusion:} When the cost lacks this origin smoothness (e.g., $h(d) = d^p/p$ for $1 < p < 2$), the singularity cannot be analytically smoothed. Instead, we must spatially exclude the collision regions to restore Lipschitz bounds. To ensure these discarded singular sets carry no mass, we adapt the Euclidean assumption to require \emph{all marginals $\mu_2, \ldots, \mu_n$ to be absolutely continuous}.
\end{enumerate}
By synthesizing these tools, we establish the absolute continuity of generalized Wasserstein barycenters on Riemannian manifolds in full generality.

\begin{thm*}
	Let $(M, \matheuvm{g})$ be a complete Riemannian manifold.
	Suppose that $h: [0, \infty) \to[0, \infty)$
    is $\mathcal{C}^2$ on $(0, +\infty)$, strictly increasing and
    strictly convex with $h(0) = 0$ and $h^{\prime}(0): = \lim_{t \downarrow 0} h(t) / t = 0$.
	Given an integer $n \ge 2$,
	let $\mu_1, \ldots, \mu_n$ be $n$ probability measures
	with positive weights $\lambda_i > 0$ such that
    $\int_{M} h(d_{\matheuvm{g}}(x_0, x)) \diff \mu_i < +\infty$
    for some point $x_0 \in M$ and $\sum_{i=1}^n \lambda_i = 1$.
	If $\mu_1$ is absolutely continuous with respect to $\operatorname{Vol}$,
	then $\mathbb{P} := \sum_{i=1}^n \lambda_i \, \delta_{\mu_i}$ has
    a unique generalized Wasserstein barycenter $\mu_\mathbb{P}$,
    and it is also absolutely continuous in one of the following two different cases:
	\begin{enumerate}
        \item $h$ is $\mathcal{C}^2$-smooth on $[0, +\infty)$, i.e., $h^{\prime\prime}(0) := \lim_{t \downarrow 0} \frac{h^\prime(t)}{t} = \lim_{t \downarrow 0} h^{\prime\prime}(t)$;
		\item measures $\mu_2, \ldots, \mu_n$ are absolutely continuous.
	\end{enumerate}
\end{thm*}

The paper is organized as follows.
Notation and preliminaries are introduced in \cref{sec:preliminaries}.
In \cref{sec:optimal_transport_general_cost}, for self-completeness, we provide a detailed proof of McCann's theorem \cite[Theorem 13]{mccann2001polar} for general costs $h(d)$ on complete Riemannian manifolds.
In \cref{sec:multi_marginal}, we discuss the properties of multi-marginal plans, specifically the injectivity of the barycenter map and the avoidance of the cut locus.
\Cref{sec:absolute_continuity} is the core proof, where we apply the decomposition strategy to establish absolute continuity via approximation.

\section{Preliminaries}

\label{sec:preliminaries}

In this section, we establish the preliminary results and notation required for our analysis.
We shall briefly review the necessary background in Riemannian geometry and measure theory.
Moreover, we specify the assumptions on the cost function $h$,
and establish some topological and measurable properties of the associated generalized barycenters.

\hspace{1cm}

Throughout this article, $(M, \matheuvm{g})$ is always a complete Riemannian manifold without boundary, not necessarily compact.
We denote by $d_{\matheuvm{g}}(\cdot, \cdot)$ the intrinsic geodesic distance on $M$.
For a point $x \in M$, $T_x M$ denotes the tangent space at $x$, and $\exp_x: T_x M \to M$ is the exponential map.
The cut locus of $x$ is denoted by $\operatorname{Cut}(x)$ \cite[Definition 4.3]{sakai1996riemannian}.
We denote by $\nabla$ the gradient operator associated with the metric $\matheuvm{g}$.
$\operatorname{Vol}$ denotes the Riemannian volume measure on $M$.

For maps $F: M \rightarrow M^\prime$ between two manifolds, we denote
by $\di F: TM \rightarrow TM^\prime$ its differential.
For a $\mathcal{C}^2$ map $f: M \rightarrow \mathbb{R}$,
its Hessian at $x \in M$, denoted by $\operatorname{Hess}_x f: T_xM \times T_xM \rightarrow \mathbb{R}$,
is the covariant derivative of $\di f$ at $x$ \cite[Proposition 2.2.6]{petersen2016riemannian}.

\subsection{The cost function}

We consider cost functions generated by a profile function $h$.
The following assumptions, adapted from \cite{kim2015multi} and \cite{mccann2001polar},
are imposed on $h$ throughout the paper.

\begin{assum*}
	Let $h: [0, \infty) \to [0, \infty)$ be a function satisfying the following assumptions,
	\begin{enumerate}[label=({H\arabic*})]
		\item
		      \label{assumption:h_at_zero}
		      $h(0) = 0$;
		\item
		      \label{assumption:h_C_2_smooth}
		      $h$ is $\mathcal{C}^2$ on $(0, \infty)$
		      and $h^\prime(0) : = \lim_{t \downarrow 0} \frac{h(t)}{t} = 0$;
		\item
		      \label{assumption:h_strict_convexity}
		      $h$ is strictly increasing ($h^{\prime}(t) > 0$) and strictly convex ($h^{\prime \prime}(t) > 0$) for $t > 0$.
	\end{enumerate}
\end{assum*}

We define the cost function $c: M \times M \to [0, \infty)$ by
\begin{equation}
	\label{defn:function_c}
	c(x, y) := h(d_{\matheuvm{g}}(x, y)).
\end{equation}
While the classical case $h(d) := d^2/2$ is smooth at the diagonal $x=y$, general costs $h(d)$ (such as $d^p$ for $1 < p < 2$) may lack smoothness at the diagonal.
However, away from the cut locus pairs and the diagonal, $c$ inherits the regularity of the distance function \cite[Proposition 4.8]{sakai1996riemannian}.

\subsection{Barycenters on proper metric spaces}

Since $(M, \matheuvm{g})$ is a complete Riemannian manifold, the metric space $(M, d_{\matheuvm{g}})$ is proper (i.e., closed bounded sets are compact).
To clarify the proofs,
we shall state the following definitions and properties for general proper metric spaces.

\begin{defn}[$h$-barycenters and their configurations for finitely many points]
    \label{defn:barycenter_configuration}
	Let $(E, d)$ be a proper metric space
	and let $\mu$ be a probability measure on $E$
	such that $\int_E h(d(x_0, x)) \diff \mu(x) < \infty$ for some $x_0 \in E$.
	A point $z \in E$ is called an \emph{$h$-barycenter} (or simply, barycenter) of $\mu$
    if it minimizes the following function over $y \in E$,
	\[
		y \mapsto \int_E h(d(y, x)) \diff \mu(x).
	\]

	In the specific case where $\mu = \sum_{i=1}^n \lambda_i\, \delta_{x_i}$ is a finitely supported
	probability measure with weights $\lambda_i > 0$,
	we denote its set of barycenters by
	\[
		\operatorname{bary}_{\lambda_1, \ldots, \lambda_n}(\{ (x_1, \ldots, x_n) \}).
	\]
	We call $\matheuvm {x} := (x_1, \ldots, x_n)$ with the associated weights
	a (barycenter) \emph{configuration}.
	The weights $\{ \lambda_i \}_{1 \le i \le n}$ under consideration
	are commonly clear from the context, for which we also employ
	the simplified symbol $\operatorname{bary}(\{\matheuvm{x}\})$.
	For a given subset $A \subset E^n$, we further define
	\[
		\operatorname{bary}(A) := \{ z \in E \mid
		z \text{ is a $h$-barycenter of } \sum_{i=1}^n \lambda_i\, \delta_{x_i}
		\text{ and }  \matheuvm{x} = (x_1, \ldots, x_n) \in A \}.
	\]
\end{defn}

The topological properties of the set of barycenters established in \cite{ma2025absolute} for the quadratic cost extend to the general cost $h(d)$.
Unlike the quadratic case, we do not rely on the metric structure of the Wasserstein space $\mathcal{W}_2(E)$, but derive these properties directly from the continuity and convexity of $h$ and the properness of the underlying metric space.

Recall that $h: [0, \infty) \to [0, \infty)$ satisfies \ref{assumption:h_at_zero}-\ref{assumption:h_strict_convexity}.
We remark that \ref{assumption:h_strict_convexity} implies the \emph{coercivity} of $h$, i.e.,
$\lim_{t \to \infty} h(t) = \infty$.
Indeed, as $h^{\prime \prime} > 0$,
\begin{equation}
	\label{equa:coercivity_implied_by_h3}
	h(t) \ge h'(1)(t-1) + h(1), \quad \forall \, t \ge 1,
\end{equation}
which further implies the coercivity as $h'(1) > 0$.

\begin{lem}[Compactness and measurable selection of barycenters]
	\label{lem:barycenter_properties}
	Let $(E, d)$ be a proper metric space.
    Fix $n \ge 2$ positive weights $\lambda_i > 0$ satisfying $\sum_{i=1}^n \lambda_i = 1$.
	Define the objective function $\Phi: E^n \times E \to [0, \infty)$ by
	\[
		\Phi(\matheuvm{x}, y) := \sum_{i=1}^n \lambda_i\,
		h(d(y, x_i)), \quad \text{for } \matheuvm{x} = (x_1, \ldots, x_n) \in E^n.
	\]
	The following five statements are true.
	\begin{enumerate}
		\item The function $m(\matheuvm{x}) := \inf_{y \in E} \Phi(\matheuvm{x}, y)$ is continuous on $E^n$.
		\item For any $\matheuvm{x} = (x_1, \ldots, x_n) \in E$, the set of barycenters
		      $\operatorname{bary}(\{ \matheuvm{x} \})$ is non-empty.
		\item The set of minimizers $\Gamma := \{ (\matheuvm{x}, y) \in E^n \times E \mid \Phi(\matheuvm{x}, y) = m(\matheuvm{x}) \}$ is a closed subset of $E^n \times E$.
		\item For any compact subset $A \subset E^n$, the set of all barycenters
		      \[
			      \operatorname{bary}(A) = \{ y \in E \mid \exists\, \matheuvm{x} \in A, (\matheuvm{x}, y) \in \Gamma \}
		      \]
		      is compact.
		\item There exists a Borel measurable map $B: E^n \to E$ such that $B(\matheuvm{x})$ is a barycenter of $\sum \lambda_i \delta_{x_i}$ for every $\matheuvm{x} \in E^n$.
	\end{enumerate}
\end{lem}

\begin{proof}
	First, we establish a uniform coercivity property.
	Let $A \subset E^n$ be a bounded set. Fix a reference point $o \in E$. There exists $R > 0$ such that $d(x_i, o) \le R$ for all $\matheuvm{x} \in A$ and all $i$.
	By the triangle inequality, $d(x_i, y) \ge d(y, o) - R$.
	Thus, $\Phi(\matheuvm{x}, y) \ge \sum \lambda_i\, h(\max(0, d(y, o) - R))$.
	As \ref{assumption:h_strict_convexity} and (\ref{equa:coercivity_implied_by_h3})
	imply that $\lim_{t \rightarrow \infty} h(t) = \infty$, we have
	\[
		\Phi(\matheuvm{x}, y) \to \infty \text{ as } d(y, o) \to \infty,
		\quad \text{uniformly for $\matheuvm{x} \in A$}.
	\]
	Consequently, for any $l \in \mathbb{R}$,
	the projection of the set $\{(\matheuvm{x}, y) \in A \times E \mid \Phi(\matheuvm{x}, y) \le l \}$
	onto $E$ is bounded.
	Since $E$ is proper, the closure of this projection is compact.

	Proof of Statement 1.
	Since $h$ is continuous, $\Phi$ is continuous.
	Fix $\matheuvm{x}_0 \in E^n$. Let $U$ be a compact neighborhood of $\matheuvm{x}_0$.
	Let $y_0$ be an arbitrary point. We have
	\[
		\forall\, \matheuvm{x} \in U,\quad
		m(\matheuvm{x}) \le \Phi(\matheuvm{x}, y_0) \le
		\sup_{\matheuvm{a} \in U} \Phi(\matheuvm{a}, y_0) =: C.
	\]
	By the uniform coercivity established above,
	there exists a compact set $K \subset E$ such that for any $\matheuvm{x} \in U$,
	if $\Phi(\matheuvm{x}, y) \le C$, then $y \in K$.
	Therefore, for $\matheuvm{x} \in U$, $m(\matheuvm{x}) = \inf_{y \in K} \Phi(\matheuvm{x}, y)$.
	Since $\Phi$ is continuous on the compact set $U \times K$,
	the function $\Phi(\cdot , y)$ defined on $U$ is thus uniformly continuous for $y \in K$.
	It follows that the marginal function $m(\matheuvm{x})$ is continuous on $U$, and thus on $E^n$.

	Proof of Statement 2.
	Fix an arbitrary point $\matheuvm{x} \in E^n$.
	Consider the function $f_{\matheuvm{x}}: y \mapsto \Phi(\matheuvm{x}, y)$ defined on $E$.
	It follows from the (uniform) coercivity of $\Phi$ implies that,
	for any $y_0 \in E$, the sub-level set
	$K = \{ y \in E \mid f_{\matheuvm{x}}(y) \le f_{\matheuvm{x}}(y_0) \}$ is bounded.
	Since $E$ is a proper metric space, $K$ is compact.
	By the Weierstrass theorem, the continuous function $f_{\matheuvm{x}}$ attains its global minimum on the compact set $K$.
	Thus, $\operatorname{bary}(\{ \matheuvm{x} \})$ is non-empty.

	Proof of Statement 3.
	The graph $\Gamma$ is the zero level set of the function $(\matheuvm{x}, y) \mapsto \Phi(\matheuvm{x}, y) - m(\matheuvm{x})$.
	Since both $\Phi$ and $m$ are continuous, this function is continuous, and thus $\Gamma$ is a closed set.

	Proof of Statement 4.
	Let $A \subset E^n$ be compact.
	The set of barycenters is the projection $\operatorname{bary}(A) = p_2(\Gamma \cap (A \times E))$, where $p_2: A \times E \to E$ is the canonical projection.
	Since $A$ is compact, the projection map $p_2$ is a \emph{closed map} \cite[Proposition 8.2]{bredon2013topology}, i.e.,
	$p_2$ maps closed sets to closed sets.
	Since $\Gamma$ is closed, $\Gamma \cap (A \times E)$ is closed in $A \times E$.
	Therefore, its image $\operatorname{bary}(A)$ is a closed subset of $E$.
	Furthermore, the function $m(\matheuvm{x})$ attains a maximum $C$ on the compact set $A$.
	Thus, if $y \in \operatorname{bary}(A)$, there exists $\matheuvm{x} \in A$ such that $\Phi(\matheuvm{x}, y) = m(\matheuvm{x}) \le C$.
	By the uniform coercivity property, the set of such $y$ is bounded.
	Since $\operatorname{bary}(A)$ is closed and bounded in a proper metric space, it is compact.

	Proof of Statement 5.
	We apply the Kuratowski and Ryll-Nardzewski measurable selection theorem (cf.\@ \cite[Theorem 2.4]{ma2025absolute}).
	We must show that the set-valued map
	$\Psi(\matheuvm{x}) = \{y \in E \mid (\matheuvm{x}, y) \in \Gamma\}$
	is weakly measurable, i.e.,
	for any open set $U \subset E$,
	the set $\{\matheuvm{x} \mid \Psi(\matheuvm{x}) \cap U \neq \emptyset\}$ is a Borel set.
	Since $E$ is a proper metric space, any open set $U$ is a countable union of compact sets $K_j$
	\cite[proof of Lemma 2.6]{ma2025absolute} \cite[Lemma 1.6]{ma2025thesis}.
	Thus, it suffices to prove that for any compact set $K \subset E$, the set $S_K = \{\matheuvm{x} \mid \Psi(\matheuvm{x}) \cap K \neq \emptyset\}$ is Borel.
	Observe that $S_K = p_1(\Gamma \cap (E^n \times K))$, where $p_1: E^n \times K \to E^n$ is the projection.
	Since $K$ is compact, the map $p_1$ is a closed map.
	As $\Gamma$ is closed, $\Gamma \cap (E^n \times K)$ is thus closed.
	Therefore, the image $S_K$ is a closed subset of $E^n$, and hence Borel.
	Since $\Psi$ has non-empty closed values and satisfies the measurability condition, a measurable selection $B$ exists.
\end{proof}

\subsubsection*{Notation of measures on Polish spaces}

Let $(E, d)$ be a Polish (i.e., complete and separable) metric space.
We denote by $\mathcal{B}(E)$ the Borel $\sigma$-algebra of $E$ and by $\mathcal{P}(E)$ the set of all probability measures on $E$.
The support of a probability measure $\mu \in \mathcal{P}(E)$,
denoted by $\operatorname{supp}(\mu)$, is the smallest closed set to which $\mu$ assigns full mass.

\section{Optimal transport with general cost functions}

\label{sec:optimal_transport_general_cost}

We begin by briefly recalling the optimal transport problem of Monge and Kantorovich, which provides the foundation for the Wasserstein distances.
Given two probability measures $\mu$ and $\nu$ on a metric space $(E, d)$,
the Monge problem seeks a transport map $T: E \to E$ pushing $\mu$ forward to $\nu$ (i.e., $T_{\#} \mu = \nu$) that minimizes the total transportation cost $\int_E c(x, T(x)) \diff \mu(x)$.
The Kantorovich relaxation generalizes this by considering transport plans $\gamma \in \Pi(\mu, \nu)$, which are probability measures on $E \times E$ with marginals $\mu$ and $\nu$ in this order.
The optimal transport cost is defined as:
\begin{equation}
    \label{equa:transport_cost}
	\mathcal{T}_c(\mu, \nu) := \inf_{\gamma \in \Pi(\mu, \nu)} \int_{E \times E} c(x, y) \diff \gamma(x, y).
\end{equation}
When the cost is a power of the distance, $c(x,y) = d_{\matheuvm{g}}(x,y)^p$, the $p$-th root of this value defines the $p$-Wasserstein distance $W_p(\mu, \nu)$.
Classical references on this topic are \cite{villani2009optimal}, \cite{Santambrogio2015}, and \cite{villani2021topics}.

Back to our specific setting of a complete Riemannian manifold $(M, \matheuvm{g})$,
for notational convenience,
we denote by $\mathcal{W}_h(M)$ the subset of $\mathcal{P}(M)$ with finite $h$-moment, i.e.,
\begin{equation}
	\label{equa:defn_W_h}
	\mathcal{W}_h(M) := \left\{ \mu \in \mathcal{P}(M) \;\middle|\;
	\int_M h(d_{\matheuvm{g}}(x_0, x)) \diff \mu(x) < \infty \text{ for some } x_0 \in M \right\}.
\end{equation}
Since $h$ is increasing \ref{assumption:h_strict_convexity}, the triangle inequality implies that the above definition of $\mathcal{W}_h(M)$ is independent of $x_0$.
As we primarily work with compactly supported measures in the subsequent sections to establish regularity, the topological subtleties of $\mathcal{W}_h$ are not central to our proofs.

In this section, we discuss the existence and uniqueness of optimal transport maps for the cost function $c(x, y) = h(d_{\matheuvm{g}}(x, y))$. The celebrated work of McCann \cite{mccann2001polar} establishes this primarily for the quadratic cost $d^2/2$, while \cite[Theorem 13]{mccann2001polar} formulates the corresponding result for general strictly convex costs between compactly supported measures. For the sake of self-completeness, we provide here a detailed proof of this theorem, explicitly detailing the arguments for our specific setting.

\subsection{\texorpdfstring{$c$}{c}-concavity}

The structure of optimal transport maps is governed by the notion of $c$-concavity.
We adopt the notation $\mathcal{I}^c(X, Y)$ \cite[Definition 3.1]{cordero2001riemannian} for the set of $c$-concave functions defined on a domain $X$ relative to a target set $Y$.

\begin{defn}[$c$-transforms and $c$-concave functions]
	Let $(M, \matheuvm{g})$ be a Riemannian manifold.
	Let \( X \) and \( Y \) be two non-empty compact subsets of \( M \).
	A function \( \psi \) : \( X \rightarrow \mathbb { R }\)
	is $c$-concave if there exists a function
	\( \xi : Y \rightarrow \mathbb { R } \) such that
	\begin{equation}
		\label{defn:c_transform}
        \psi ( x ) = \inf _ { y \in Y } \{ c ( x , y ) - \xi ( y ) \}, \quad \forall x \in X.
	\end{equation}
	We write it as \( \psi = \xi ^ { c } \) and call $\psi$ the $c$-transform of $\xi$.
	The set of all $c$-concave functions with respect to $X$ and $Y$
	is denoted by $\mathcal{I}^c(X, Y)$.
\end{defn}

Since the cost function $c(x, y) = h(d_{\matheuvm{g}}(x, y))$ involves the Riemannian distance, it is generally not smooth everywhere due to the presence of the cut locus.
However, distance functions on manifolds possess a one-sided regularity known as super-differentiability.
Recall that a function $f: M \to \mathbb{R}$ is \emph{super-differentiable} at $x$ if there exists a vector $v \in T_x M$, called a \emph{super-gradient}, such that
\begin{equation}
    \label{equa:defn_super_differentiable}
	f(\exp_x u) \le f(x) + \matheuvm{g}_x(v, u) + o(\|u\|) \quad \text{as } u \to 0 \text{ in } T_x M.
\end{equation}
We denote by $\partial^+ f(x)$ the set of super-gradients at $x$.
The following chain rule allows us to transfer the regularity of the distance to the cost,
thanks to a standard first-order Taylor expansion of $g$ combined with the super-differentiability inequality for $f$,
whose direction is further persevered as $g$ is assumed to be non-decreasing.

\begin{lem}[Chain rule for super-gradients, \citeinfo{Lemma 5}{mccann2001polar}]
	\label{lem:chain_rule}
	Let $(M, \matheuvm{g})$ be a Riemannian manifold.
    Fix a point $x \in M$.
	Let $f: M \to \mathbb{R}$ be super-differentiable at $x$ with $v \in \partial^+ f(x)$.
	Let $g: \mathbb{R} \to \mathbb{R}$ be a non-decreasing function that is differentiable at $f(x)$.
	Then $g \circ f$ is super-differentiable at $x$, and $g'(f(x))v \in \partial^+ (g \circ f)(x)$.
\end{lem}

We now establish the regularity of the cost function $x \mapsto c(x, y)$ for a fixed $y$.
Crucially, while the squared distance $d^2$ is smooth at the diagonal $x=y$, the general cost $h(d)$ may not be.
However, away from the diagonal, the cost inherits the super-differentiability of the distance function.

\begin{lem}[Super-differentiability of the cost]
	\label{lem:cost_super-diff}
	Let $(M, \matheuvm{g})$ be a complete Riemannian manifold.
	Fix a point $y \in M$.
	The function $\psi_y(x) := h(d_{\matheuvm{g}}(x, y))$ is super-differentiable at every point $x \in M \setminus \{y\}$.
	If $\sigma: [0, d_{\matheuvm{g}}(x,y)] \to M$ is any unit-speed minimal geodesic from $x$ to $y$, then
	\[
		v := - h'(d_{\matheuvm{g}}(x, y)) \dot{\sigma}(0) \in \partial^+ \psi_y(x).
	\]
\end{lem}

\begin{proof}
	We apply the chain rule (\Cref{lem:chain_rule}) to the distance function $f(x) := d_{\matheuvm{g}}(x, y)$.
	It is a classical result in Riemannian geometry that the distance function to a fixed point $y$ is super-differentiable on $M \setminus \{y\}$, as discussed in \cite[Proposition 6]{mccann2001polar}.
	Here, we provide a proof based on the first variational formula of arc length.

	Fix $x \in M \setminus \{y\}$. Let $\sigma: [0, \ell] \to M$ be a unit-speed minimal geodesic from $x$ to $y$, where $\ell := d_{\matheuvm{g}}(x, y)$.
	Let $u := \dot{\sigma}(0) \in T_x M$ be the initial velocity.
	To test super-differentiability, consider an arbitrary vector $w \in T_x M$.
	We construct a smooth variation $\alpha: [0, \ell] \times (-\epsilon, \epsilon) \to M$ with:
	\begin{enumerate}
		\item $\alpha(t, 0) = \sigma(t)$ for all $t \in [0, \ell]$;
		\item $\alpha(0, s) = \exp_x(s w)$ for $s \in (-\epsilon, \epsilon)$;
		\item $\alpha(\ell, s) = y$ for all $s \in (-\epsilon, \epsilon)$.
	\end{enumerate}
	Let $L(s) := \int_0^\ell \left\| \frac{\partial \alpha}{\partial t}(t, s) \right\| \diff t$ be the length of the curve $t \mapsto \alpha(t, s)$.
	By the definition of the Riemannian distance, we have the inequality
	\begin{equation}
		\label{eq:dist_length_ineq}
		d_{\matheuvm{g}}(\exp_x(s w), y) \le L(s).
	\end{equation}
	Let $V(t) := \frac{\partial \alpha}{\partial s}(t, 0)$ be the variational vector field along $\sigma$.
	From our construction, the boundary conditions are $V(0) = w$ and $V(\ell) = 0$.
	The first variation formula of arc length
	\cite[Theorem II.4.1]{chavel2006riemannian} \cite[(1.5) in Chapter 1]{cheeger2008comparison} gives:
	\[
		L'(0) = \left. \langle V(t), \dot{\sigma}(t) \rangle \right|_0^\ell - \int_0^\ell \langle V(t), \nabla_{\dot{\sigma}} \dot{\sigma}(t) \rangle \diff t.
	\]
	Since $\sigma$ is a geodesic, $\nabla_{\dot{\sigma}} \dot{\sigma} \equiv 0$. Substituting the boundary values, we obtain:
	\[
		L'(0) = \langle 0, \dot{\sigma}(\ell) \rangle - \langle w, \dot{\sigma}(0) \rangle = - \langle w, u \rangle.
	\]
	Combining this with \eqref{eq:dist_length_ineq}, we have
	\[
		d_{\matheuvm{g}}(\exp_x(s w), y) \le L(0) + s L'(0) + o(s) = d_{\matheuvm{g}}(x, y) - \matheuvm{g}_x(u, s w) + o(s).
	\]
	This confirms that $f(x)$ is super-differentiable with super-gradient $-u$.
	Finally, since $h$ satisfies the conditions of \Cref{lem:chain_rule}, $\psi_y = h \circ f$ is super-differentiable at $x$, and the vector $h'(f(x))(-u) = - h'(d_{\matheuvm{g}}(x, y)) \dot{\sigma}(0)$ is a super-gradient.
\end{proof}

To establish the existence of optimal maps, we examine the regularity of $c$-concave functions.

\begin{lem}[Lipschitz continuity of $c$-concave functions]
	\label{lem:lipschitz_potentials}
    Fix a compact set $Y \subset M$ 
    and a function $\zeta: Y \to \mathbb{R}$ defined on it.
    Define the function $\psi: M \to \mathbb{R}$ by
	\begin{equation}
        \label{equa:extended_c_concave_function}
		\psi(x) := \inf_{y \in Y} \{ h(d_{\matheuvm{g}}(x, y)) - \zeta(y) \},
        \quad x \in M.
	\end{equation}
	Then $\psi$ is locally Lipschitz continuous on $M$.
	In particular, for any compact subset $X \subset M$,
    $\psi^{\prime}(x)$ exists for $\operatorname{Vol}$-almost every $x \in X$
    and the restriction $\psi|_X$, belonging to the set $\mathcal{I}^c(X, Y)$,
    is Lipschitz continuous.
\end{lem}

\begin{proof}
	Fix an arbitrarily chosen non-empty compact set $X \subset M$.
	Since $X$ and $Y$ are compact, the set of distances $\{ d_{\matheuvm{g}}(x, y) \mid x \in X, y \in Y \}$ is contained in a compact interval $[0, R]$ for some $R > 0$.
	We define the constant $L := h'(R)$.
	Since the derivative of $h$ is non-decreasing,
    $h$ is $L$-Lipschitz continuous on $[0, R]$
    thanks to the intermediate value theorem.
	
	Consider the family of functions $f_y(x) := h(d_{\matheuvm{g}}(x, y)) - \zeta(y)$ for $y \in Y$.
	For any $x, z \in X$ and any fixed $y \in Y$, we have
	\[
		|f_y(x) - f_y(z)| = |h(d_{\matheuvm{g}}(x, y)) - h(d_{\matheuvm{g}}(z, y))| \le L |d_{\matheuvm{g}}(x, y) - d_{\matheuvm{g}}(z, y)| \le L d_{\matheuvm{g}}(x, z),
	\]
	where the last inequality follows from the triangle inequality for the metric $d$.
	Thus, the family $\{f_y\}_{y \in Y}$ is uniformly $L$-Lipschitz on $X$.
	
	Now consider $\psi(x) = \inf_{y \in Y} f_y(x)$.
	Fix $x, z \in X$. For any $\epsilon > 0$, there exists $y_z \in Y$ such that $f_{y_z}(z) < \psi(z) + \epsilon$.
	Then
	\[
		\psi(x) \le f_{y_z}(x) \le f_{y_z}(z) + L d_{\matheuvm{g}}(x, z) < \psi(z) + L d_{\matheuvm{g}}(x, z) + \epsilon.
	\]
	Since $\epsilon$ is arbitrary, $\psi(x) - \psi(z) \le L d_{\matheuvm{g}}(x, z)$.
	Interchanging $x$ and $z$ yields $|\psi(x) - \psi(z)| \le L d_{\matheuvm{g}}(x, z)$.
	Therefore, $\psi$ is locally Lipschitz continuous on $M$.
    By Rademacher's theorem \cite[Lemma 4]{mccann2001polar}, $\psi$ is differentiable $\operatorname{Vol}$-almost everywhere.
\end{proof}

\begin{rmk}[$c$-concavity and semi-concavity]
	\label{rmk:lack_of_semi_concavity}
	Recall from \cite[Definition 3.4]{ma2025absolute} that a function $\psi: O \to \mathbb{R}$ defined on an open subset $O \subset M$ is \emph{locally semi-concave} if for every $x \in O$, there exists a geodesically convex neighborhood $U \subset O$ and a smooth function $V: U \to \mathbb{R}$ such that $\psi + V$ is geodesically concave on $U$.
	This property is highly desirable for non-smooth analysis and optimization because semi-concave functions admit a second-order expansion almost everywhere (Alexandrov's theorem \cite[Proposition 3.10]{ma2025absolute}).

	However, for general cost functions $h(d)$, $c$-concave potentials may fail to be semi-concave.
	Consider the case of the Euclidean space $\mathbb{R}^m$ with cost $h(d) = d^p/p$ for $1 < p < 2$.
	The Hessian of this function scales as $|x - y|^{p-2}$.
	Since $p < 2$, the eigenvalues of the Hessian tend to $+\infty$ as $x \to y$.
	This contradicts the definition of semi-concavity, which requires the Hessian to be bounded from above.
	Consequently, the potentials associated with $p$-Wasserstein metrics for $p < 2$ are not semi-concave.
    Note that if we assume stronger regularity (cf.\@\cite[Corollary C.5]{gangbo1996geometry}), specifically $h \in \mathcal{C}^2([0, \infty))$ (e.g., $p \ge 2$), one can prove semi-concavity by performing a Taylor expansion and exploiting the semi-concavity of the squared distance function (cf.\@ \cite[Lemma 3.11, Lemma 3.12]{cordero2001riemannian} and \cite[Theorem 6.1]{kim2015multi}).
\end{rmk}

\subsection{Optimal transport map for general costs}

The core of the Brenier-McCann theorems is relating the optimal transport map to the gradient of the potential.
The following ``Tangency Lemma'' generalizes \cite[Lemma 7]{mccann2001polar} to our cost $h(d)$.

\begin{lem}[Tangency Lemma]
	\label{lem:tangency}
	Let $(M, \matheuvm{g})$ be a Riemannian manifold.
    Fix two compact subsets $X, Y \subset M$.
    Let $\psi \in \mathcal{I}^c(X, Y)$ be a $c$-concave function
    and denote by $\psi^c \in \mathcal{I}^c(Y, X)$ its $c$-transform.
	Suppose that $\psi$ is differentiable at an interior point $x \in \mathring X$.
	If there exists $y \in Y$ ($y \neq x$) such that
	\begin{equation}
		\label{eq:contact}
		\psi(x) + \psi^c(y) = c(x, y),
	\end{equation}
	then the gradient $\nabla \psi(x) \neq 0 $ and
    $y$ is uniquely determined by it via the formula:
	\begin{equation}
		\label{eq:optimal_map_formula}
		y = \exp_x \left( - (h')^{-1}(\|\nabla \psi(x)\|) \frac{\nabla \psi(x)}{\|\nabla \psi(x)\|} \right).
	\end{equation}
\end{lem}
\begin{proof}
    Consider the function $F(z) := c(z, y) - \psi(z)$ defined on $X$.
    By the definition of the $c$-transform, $\psi^c(y) = \inf_{z \in X} F(z)$.
    The equality hypothesis \eqref{eq:contact} implies $F(x) = \psi^c(y)$; hence, $F$ attains a global minimum at $x$.
    Let $v \in T_x M$ be an arbitrary unit vector.
    Using the Taylor expansion of $\psi$ and the super-differentiability of the cost function (\Cref{lem:cost_super-diff}), we analyze the behavior of $F$ along the geodesic $\gamma(t) = \exp_x(t v)$ for small $t > 0$.
    
    Let $\sigma$ be any unit-speed minimal geodesic connecting $x$ to $y$.
    By \Cref{lem:cost_super-diff}, the vector $\xi := -h'(d_{\matheuvm{g}}(x,y)) \dot{\sigma}(0)$ is a super-gradient of $c(\cdot, y)$ at $x$.
    By definition \eqref{equa:defn_super_differentiable}, we have the upper bound
    \[
        c(\exp_x(tv), y) \le c(x, y) + t \matheuvm{g}_x( \xi, v ) + o(t).
    \]
    Conversely, the differentiability of $\psi$ at the interior point $x$ yields
    \[
        \psi(\exp_x(tv)) = \psi(x) + t \matheuvm{g}_x( \nabla \psi(x), v ) + o(t).
    \]
    Substituting these expansions into the minimality condition $F(\exp_x(tv)) \ge F(x)$, we obtain
    \[
        (c(x, y) + t \matheuvm{g}_x( \xi, v ) ) -
        (\psi(x) + t \matheuvm{g}_x( \nabla \psi(x), v )) + o(t) \ge c(x, y) - \psi(x).
    \]
    Simplifying and dividing by $t$ yields
    \[
        \matheuvm{g}_x( \xi - \nabla \psi(x), v ) \ge - \frac{o(t)}{t}.
    \]
    Letting $t \to 0$, we deduce $\matheuvm{g}_x( \xi - \nabla \psi(x), v ) \ge 0$.
    Since $v \in T_x M$ is arbitrary, we may replace $v$ with $-v$ to conclude that equality must hold:
    \begin{equation}
        \label{eq:grad_ident}
        \nabla \psi(x) = \xi = -h'(d_{\matheuvm{g}}(x, y)) \dot{\sigma}(0).
    \end{equation}
    
    This identity uniquely determines $y$. Taking the norm on both sides of \eqref{eq:grad_ident}, and recalling that $\|\dot{\sigma}(0)\|=1$ and $h'$ is positive on $(0, \infty)$ by \ref{assumption:h_C_2_smooth} and \ref{assumption:h_strict_convexity}, we recover the distance:
    \[
        \|\nabla \psi(x)\| = h'(d_{\matheuvm{g}}(x, y)) \implies d_{\matheuvm{g}}(x, y) = (h')^{-1}(\|\nabla \psi(x)\|),
    \]
    where the strict convexity \ref{assumption:h_strict_convexity} of $h$ ensures $h'$ is invertible.
    Note that \eqref{eq:grad_ident} implies $\nabla \phi (x) \neq 0$,
    and thus determines the unique initial direction of the minimal geodesic:
    \[
        \dot{\sigma}(0) = - \frac{\nabla \psi(x)}{\|\nabla \psi(x)\|}.
    \]
    Consequently, $y = \exp_x( d_{\matheuvm{g}}(x, y) \dot{\sigma}(0) )$ is uniquely determined by $\nabla \psi(x)$, proving \eqref{eq:optimal_map_formula}.
\end{proof}

\begin{rmk}
    The above proof shows that the first-order differentiability of $\psi$ at $x$ implies the uniqueness of the geodesic from $x$ to $y$, even in the case $y \in \operatorname{Cut}(x)$.
    Recall that in \cite[Proposition 4.1 (a)]{cordero2001riemannian}, it is shown for the quadratic case $h(d) = d^2 / 2$ that second-order differentiability of $\psi$ at $x$ implies $y \notin \operatorname{Cut}(x)$. Their argument relies on the property that the squared distance function fails to be semi-convex at the cut locus \cite[Proposition 2.5]{cordero2001riemannian}.
    This cut-locus avoidance property still holds for our general cost function. Indeed, the chain rule can be applied to the composition $h(d) = h(\sqrt{d^2})$. Since the distance $d_{\matheuvm{g}}(x, y)$ is strictly positive for any $y \in \operatorname{Cut}(x)$, the square root and possible sigularity of $h$ at $0$ pose no regularity issue. The Hessian of $h(d_{\matheuvm{g}}(\cdot, y))$ therefore inherits the singular behavior of the squared distance term, leading to the same conclusion.
\end{rmk}

We are now ready to state the main theorem of this section, which guarantees the existence and uniqueness of the optimal transport plan (map).
To address the interior point condition of \Cref{lem:tangency},
we consider two bounded open sets containing the measure supports, as in \cite[Theorem 4.2]{cordero2001riemannian}.

\begin{thm}[Optimal transport map for general costs, \citeinfo{Theorem 13}{mccann2001polar}]
	\label{thm:optimal_map}
	Let $(M, \matheuvm{g})$ be a complete Riemannian manifold and let $h$ satisfy assumptions
    \ref{assumption:h_at_zero}-\ref{assumption:h_strict_convexity}.
	Let $\mu, \nu \in \mathcal{W}_h(M)$ be two probability measures 
    whose supports are contained in two bounded open sets
    $\mathcal{X}, \mathcal{Y} \subset M$ respectively.
    Assume that $\mu$ is absolutely continuous with respect to $\operatorname{Vol}$.
	Then there exists a unique optimal transport plan $\gamma$ for the cost $c(x, y) = h(d_{\matheuvm{g}}(x, y))$, and it is induced by a map $T: M \to M$, i.e., $\gamma = (\operatorname{Id}, T)_{\#} \mu$.
	The map $T$ is uniquely determined $\mu$-almost everywhere by
    a $c$-concave function $\psi \in \mathcal{I}^c(\widebar{\mathcal{X}}, \widebar{\mathcal{Y}})$,
	\[
		T(x) = \exp_x \left( - (h')^{-1}(\|\nabla \psi(x)\|) \frac{\nabla \psi(x)}{\|\nabla \psi(x)\|} \right),
	\]
    where we define by convention that $T(x) = x$ if $\nabla \psi(x) = 0$.
\end{thm}

\begin{proof}
    Since $\mu, \nu$ have compact supports and $h$ is continuous and non-negative,
    the Kantorovich duality \cite[Theorem 5.10]{villani2009optimal} implies
	the existence of an optimal plan and a $c$-concave potential function
    $\psi \in \mathcal{I}^c(\widebar {\mathcal{X}}, \widebar {\mathcal{Y}})$ maximizing the dual problem.
	It further ensures that the support of any given optimal transport plan $\gamma$
    is contained in the $c$-super-differential of $\psi$, i.e.,
	\begin{equation}
        \label{equa:support_in_c_super_differential}
		\operatorname{supp}(\gamma) \subset \partial^c \psi :=
        \{ (x, y) \in \widebar{\mathcal{X}} \times \widebar{\mathcal{Y}}
        \mid \psi(x) + \psi^c(y) = c(x, y) \}.
	\end{equation}

    Let $D \subset \mathcal{X}$ be the set where $\psi$ is differentiable.
    According to \Cref{lem:lipschitz_potentials},
    $\mu(M \setminus D) = 0$ since $\mu$ is absolutely continuous.
    For any $x \in D$, if $(x, y) \in \operatorname{supp}(\gamma)$,
    then by \eqref{equa:support_in_c_super_differential},
    the pair satisfies the equality condition of \Cref{lem:tangency}.
    If $y=x$, then $\nabla \psi(x) = 0$ (as $c(\cdot, x)$ has gradient 0 at $x$), yielding $T(x) = x = y$. If $y \neq x$, \Cref{lem:tangency} directly implies $y = T(x)$.
    Hence, we have $y = T(x)$ for $\gamma$-almost every $(x, y)$.
    The measurability of $T$ follows from the measurability of the set $D$ and the gradient map
    $\nabla \psi$ defined on it \cite[Lemma 2]{mccann2001polar}.
    Hence, the equality $\gamma = (\operatorname{Id}, T)_{\#} \mu$ holds
    for any chosen optimal transport plan $\gamma$,
    which further implies the uniqueness of $\gamma$.
\end{proof}

\section{Multi-marginal optimal transport plans}

\label{sec:multi_marginal}

In this section, we study the structure of multi-marginal optimal transport plans related to the barycenter problem.
Given a finitely supported probability measure $\mathbb{P} :=\sum_{i=1}^n \lambda_i\, \delta_{\mu_i}$
on $\mathcal{W}_h(M)$ with $\lambda_i > 0$ and $n \ge 2$,
its barycenter problem,
\begin{equation}
    \label{equa:barycenter_problem}
    \text{minimize}\quad
    \nu \mapsto \sum_{i=1}^n \lambda_i\, \mathcal{T}_c(\nu, \mu_i)
    \quad \text{over }
    \nu \in \mathcal{W}_h(M),
\end{equation}
is intimately related to the following multi-marginal optimal transport problem on $M^n$:
\begin{equation}
	\label{eq:mmot_problem}
	\inf_{\gamma \in \Pi(\mu_1, \ldots, \mu_n)} \int_{M^n} C(\matheuvm{x}) \diff \gamma(\matheuvm{x}),
\end{equation}
where $\Pi(\mu_1, \ldots, \mu_n)$ denotes the set of all probability measures on $M^n$
with marginals $\mu_1, \ldots, \mu_n$ in this given order,
and the cost function $C: M^n \to [0, \infty)$ is defined by:
\begin{equation}
	\label{eq:barycenter_cost}
	C(x_1, \ldots, x_n) := \inf_{y \in M} \sum_{i=1}^n \lambda_i\, h(d_{\matheuvm{g}}(y, x_i)).
\end{equation}
For the quadratic case $h(d) = d^2 / 2$, as shown in \cite[Proposition 4.2]{agueh2011barycenters},
\cite[Section 5]{kim2015multi} and \cite[Theorem 8]{le2017existence} for different settings
of the base metric space,
if $\gamma$ minimizes \eqref{eq:mmot_problem} and $B(\matheuvm{x})$ is a measurable selection of the minimizer in \eqref{eq:barycenter_cost}, then the push-forward $B_{\#} \gamma$ is a Wasserstein barycenter of $\mathbb{P}$.

We begin with two fundamental properties of barycenters:
their avoidance of cut loci and injectivity for configurations (\Cref{defn:barycenter_configuration})
in the support of multi-marginal optimal transport plans.
These results are based on the work of Kim and Pass \cite{kim2015multi}.

\subsection{Regularity of barycenter configurations}

The differentiability of the distance function $y \mapsto d_{\matheuvm{g}}(x, y)$
fails at $\{ x \} \cup \operatorname{Cut}(x)$ \cite[Proposition 4.8]{sakai1996riemannian}.
To employ differential tools (gradients and Hessians), it is crucial to ensure for a given
configuration $\matheuvm x = (x_1, \ldots, x_n) \in M^n$, its barycenters avoid the cut loci of the marginal points $\{x_i\}_{1 \le i \le n}$.

The following lemma extends \cite[Lemma 3.1]{kim2015multi}.
In \cite[proof of Theorem 6.1]{kim2015multi}, this property is proven by contradiction,
assuming $z \in \operatorname{Cut}(x_i)$.
Regarding the possible regularity issues of $h$ at $0$,
note that only the second-order differentiability of
the map $y \mapsto h(d_{\matheuvm{g}}(y, x_i))$ at $y = z$ is used in the proof,
which is valid since $d(z, x_i) > 0$ for $z \in \operatorname{Cut}(x_i)$.

\begin{lem}[Avoidance of the cut loci]
	\label{lem:cut_locus_avoidance}
	Let $(M, \matheuvm{g})$ be a complete Riemannian manifold.
    Let $\mu = \sum_{i=1}^n \lambda_i \,\delta_{x_i}$ with $n \ge 2$ and $\lambda_i > 0$
    be a finitely supported probability measure on $M$.
    If $z$ is a $h$-barycenter of $\mu$,
	then for every $i \in \{1, \ldots, n\}$, $z \notin \operatorname{Cut}(x_i)$.
	In particular, the function $y \mapsto h(d_{\matheuvm{g}}(y, x_i))$ is smooth in a neighborhood of $z$ for all $i$.
\end{lem}

Note that in \ref{assumption:h_C_2_smooth}, the assumption
$\lim_{t \rightarrow 0} h(t) / t = 0$ implies that the map $y \rightarrow h(d_{\matheuvm{g}}(y, x))$
is differentiable at $y = x$ with gradient $0$, by passing to the normal coordinates around $x$.
Hence, \Cref{lem:cut_locus_avoidance} allows us to use the first-order optimality condition for the barycenter, i.e., if $z$ minimizes \eqref{eq:barycenter_cost}, then
\begin{equation}
	\label{equa:first_order_condition}
	\sum_{i=1}^n \lambda_i\, \nabla_z [ h(d_{\matheuvm{g}}(z, x_i)) ] = 0.
\end{equation}
The next lemma, generalizing \cite[Lemma 3.5]{kim2015multi}, establishes that configurations
in the support of a given multi-marginal optimal transport plan
are uniquely determined by their barycenters.

\begin{lem}[Injectivity of barycenter configurations]
	\label{lem:barycenter_injectivity}
	Let $(M, \matheuvm{g})$ be a complete Riemannian manifold.
    Given a finitely supported probability measure $\mathbb{P} = \sum_{i=1}^n \lambda_i \, \delta_{\mu_i}$
    on $\mathcal{W}_h(M)$ with $\lambda_i > 0$ and $n \ge 2$,
	let $\gamma$ be a multi-marginal optimal transport plan for the cost \eqref{eq:barycenter_cost}.
	Suppose $\matheuvm{x} = (x_1, \ldots, x_n)$ and $\tilde{\matheuvm{x}} = (\tilde{x}_1, \ldots, \tilde{x}_n)$ are two points in the support of $\gamma$.
	If there exists a point $y \in M$ that is a barycenter for both configurations, i.e., $y \in \operatorname{bary}(\{\matheuvm{x}\}) \cap \operatorname{bary}(\{\tilde{\matheuvm{x}}\})$, then $\matheuvm{x} = \tilde{\matheuvm{x}}$.
\end{lem}

\begin{proof}
	Since $\matheuvm{x}, \tilde{\matheuvm{x}} \in \operatorname{supp}(\gamma)$, 
    the $c$-cyclical monotonicity of the support of optimal plans
	\cite[Proposition 2.3]{kim2014general} implies
	\[
		C(\matheuvm{x}) + C(\tilde{\matheuvm{x}}) \le C(x_1, \tilde{x}_2, \ldots, \tilde{x}_n) + C(\tilde{x}_1, x_2, \ldots, x_n).
	\]
	We fix a common minimizer $y$ and show that $x_1 = \tilde x_1$.
	Using the definition of $C$ as an infimum, we substitute $y$ into the costs on the right-hand side (which gives an upper bound):
	\begin{align*}
		C(x_1, \tilde{x}_2, \ldots, \tilde{x}_n) & \le \lambda_1 h(d_{\matheuvm{g}}(y, x_1)) + \sum_{i=2}^n \lambda_i\, h(d_{\matheuvm{g}}(y, \tilde{x}_i)), \\
		C(\tilde{x}_1, x_2, \ldots, x_n)         & \le \lambda_1 h(d_{\matheuvm{g}}(y, \tilde{x}_1)) + \sum_{i=2}^n \lambda_i\, h(d_{\matheuvm{g}}(y, x_i)).
	\end{align*}
	Summing these yields exactly $\sum_{i=1}^n \lambda_i\, h(d_{\matheuvm{g}}(y, x_i)) + \sum_{i=1}^n \lambda_i\, h(d_{\matheuvm{g}}(y, \tilde{x}_i))$, which equals $C(\matheuvm{x}) + C(\tilde{\matheuvm{x}})$ because $y$ is a minimizer for both.
	Therefore, the inequalities must be equalities.
	This implies that $y$ is also a minimizer for the mixed configuration $(x_1, \tilde{x}_2, \ldots, \tilde{x}_n)$.

	Since $y$ is a minimizer for $\matheuvm{x}$, it satisfies the critical point equation $\sum_{i=1}^n \lambda_i\, \nabla_y [ h(d_{\matheuvm{g}}(y, x_i)) ] = 0$,
	thanks to \ref{assumption:h_C_2_smooth} and \Cref{lem:cut_locus_avoidance}.
	Similarly, for the mixed configuration, we must have:
	\[
		\lambda_1 \nabla_y [ h(d_{\matheuvm{g}}(y, x_1)) ] + \sum_{i=2}^n \lambda_i\, \nabla_y [ h(d_{\matheuvm{g}}(y, \tilde{x}_i)) ] = 0.
	\]
	Comparing this with the optimality condition for $\tilde{\matheuvm{x}}$, which is $\sum_{i=1}^n \lambda_i\, \nabla_y [ h(d_{\matheuvm{g}}(y, \tilde{x}_i)) ] = 0$, we subtract the two to find:
	\begin{equation}
		\label{equa:gradient_equality}
		\lambda_1 \left( \nabla_y [ h(d_{\matheuvm{g}}(y, x_1)) ] - \nabla_y [ h(d_{\matheuvm{g}}(y, \tilde{x}_1)) ] \right) = 0.
	\end{equation}

	Consider an arbitrarily chosen point $x \in M \setminus \operatorname{Cut}(y)$.
	If $y \neq x$, then (cf.\@ \Cref{lem:cost_super-diff})
	\[
		\nabla_y [ h(d_{\matheuvm{g}}(y, x)) ] = h'(d_{\matheuvm{g}}(y, x)) \nabla_y d_{\matheuvm{g}}(y, x) = - h'(d_{\matheuvm{g}}(y, x)) \frac{\exp_y^{-1}(x)}{d_{\matheuvm{g}}(y, x)}.
	\]
	Moreover, according to \ref{assumption:h_strict_convexity},
	$\nabla_y [h(d_{\matheuvm{g}}(y, x))] = 0$ if and only if
	$d_{\matheuvm{g}}(y, x) = 0$, i.e., $y = x$.
	Hence, the point $x$ is uniquely determined by the gradient $\nabla_y [h(d_{\matheuvm{g}}(y, x))]$.
	It follows from \eqref{equa:gradient_equality} that $x_1 = \tilde{x}_1$ .
	Repeating the argument for all indices $i$ proves $\matheuvm{x} = \tilde{\matheuvm{x}}$.
\end{proof}

\begin{rmk}[Counter-examples of barycenter configuration injectivity]
	The assumption $h^\prime(0) = 0$ in \ref{assumption:h_C_2_smooth} is crucial for
	\Cref{lem:barycenter_injectivity} to hold.
	For example, consider the case $M = \mathbb{R}$, $h(t) = t^2 + t$,
	and two measures $\mu_1 = \delta_{0}$, $\mu_2 = \frac{1}{2}\delta_1 + \frac{1}{2} \delta_{0.5}$
	with coefficients $\lambda_1 = 0.8$, $\lambda_2 = 0.2$.
	We have two possible configurations $(0, 1)$ and $(0, 0.5)$, corresponding to the functions:
	\[
		F_1(y) = 0.8\, h(|y|) + 0.2\,h(|1-y|),\quad
		F_2(y) = 0.8\, h(|y|) + 0.2\,h(|0.5-y|).
	\]
	Since $h(|y|)$, $h(|1 - y|)$ and  $h(|0.5 - y|)$ are convex,
	both $F_1$ and $F_2$ are convex functions.
	Calculations show that $0$ belongs to the sub-differentials of $F_1$ and $F_2$ at $0$.
	Hence, both functions reach their global minimum at $y = 0$.
	This means two distinct configurations yield the exact same barycenter, contradicting \Cref{lem:barycenter_injectivity}.
\end{rmk}

\subsection{Properties of generalized Wasserstein barycenters}

We now leverage the structural properties of multi-marginal plans to establish the existence and uniqueness of generalized Wasserstein barycenters for finitely many measures.
We consider the probability measure
$\mathbb{P} := \sum_{i=1}^n \lambda_i\, \delta_{\mu_i}$ on $\mathcal{W}_h(M)$.
Recall from \eqref{equa:barycenter_problem} that a barycenter $\mu_{\mathbb{P}}$ of $\mathbb{P}$ is a minimizer of the functional $\nu \mapsto \sum_{i=1}^n \lambda_i\, \mathcal{T}_c(\nu, \mu_i)$.

With a measurable barycenter selection map, we can push forward the multi-marginal optimal plan to a barycenter on $M$.
We remark that
while the barycenter of a configuration $\matheuvm{x} \in M^n$ is unique for $\gamma$-almost every $\matheuvm{x}$ (\Cref{lem:barycenter_injectivity}), a globally measurable selection
is still required to define the candidate measure of generalized Wasserstein barycenter.
The following proposition generalizes the construction in \cite[Proposition 2.8]{ma2025absolute},
whose proof can be extended to justify our case
by replacing the squared distance $d(x, y)^2$ cost function with
$c(x, y) = h(d_{\matheuvm{g}}(x, y))$.

\begin{prop}[Construction of generalized Wasserstein barycenters]
	\label{prop:barycenter_construction}
	Let $(M, \matheuvm{g})$ be a complete Riemannian manifold.
    Given a finitely supported probability measure $\mathbb{P} = \sum_{i=1}^n \lambda_i \, \delta_{\mu_i}$
    on $\mathcal{W}_h(M)$ with $\lambda_i > 0$ and $n \ge 2$,
    let $\gamma \in \Pi(\mu_1, \ldots, \mu_n)$ be a solution to the multi-marginal optimal transport problem \eqref{eq:mmot_problem}.
	Let $B: M^n \to M$ be a measurable barycenter selection map, whose existence
	is guaranteed by \Cref{lem:barycenter_properties}.
	Then the measure $\mu_{\mathbb{P}} := B_{\#} \gamma$ is a generalized Wasserstein barycenter of $\mathbb{P}$.
	Moreover, for each $i \in \{1, \ldots, n\}$, the plan $\gamma_i := (B, p_i)_{\#} \gamma$ is an optimal transport plan between $\mu_{\mathbb{P}}$ and $\mu_i$ for the cost function $c$.
\end{prop}

The uniqueness of barycenters holds when one marginal is absolutely continuous.

\begin{thm}[Uniqueness of generalized Wasserstein barycenters]
	\label{thm:barycenter_uniqueness}
	Let $(M, \matheuvm{g})$ be a complete Riemannian manifold.
    Let $\mathbb{P} = \sum_{i=1}^n \lambda_i \, \delta_{\mu_i}$ be
    a finitely supported probability measure
    on $\mathcal{W}_h(M)$ with $\lambda_i > 0$ and $n \ge 2$.
	If $\mu_1$ is absolutely continuous with respect to $\operatorname{Vol}$, then the generalized Wasserstein barycenter of $\mathbb{P} : = \sum_{i=1}^n \lambda_i\, \delta_{\mu_i}$ is unique.
\end{thm}

\begin{proof}
	In the quadratic case $h(d) = d^2/2$,
	\cite[Proposition 2.10]{ma2025absolute} implies that if for all $\nu \in \mathcal{W}_h(M)$,
	any optimal transport plan from $\mu_1$ to $\nu$
	is induced by a transport map, then the barycenter $\mu_{\mathbb{P}}$ is unique.

	We first verify that \cite[Proposition 2.10]{ma2025absolute}
	remains valid for our general cost function, i.e., with $\mathcal{W}_2(M)$
	replaced by $\mathcal{W}_h(M)$.
	Indeed, its proof is based on the convexity of $\mathcal{T}_c(\mu, \nu)$ with respect to
	the following convex structure of $\mathcal{W}_h(M)$:
	for all $\nu_1, \nu_2 \in \mathcal{W}_h(M)$ and $\lambda \in (0, 1)$,
	\[
		\mathcal{T}_c(\mu, \lambda\, \nu_1 + (1-\lambda) \nu_2)
		\le \lambda\,\mathcal{T}_c(\mu, \nu_1) + (1-\lambda) \mathcal{T}_c(\mu, \nu_2).
	\]
	which still holds for the cost $c$ according to the definition \eqref{equa:transport_cost}.
	This convex inequality is enforced as strict convexity,
	provided that any optimal transport plan from $\mu$ to an arbitrarily chosen measure $\nu$
	is induced by a transport map.
	If the strict convexity is shown,
	then the uniqueness of $\mu_\mathbb{P}$ follows from classical arguments in convex analysis.
	See the proof of \cite[Proposition 2.10]{ma2025absolute} for details.

	We are thus left to show the property that any optimal transport plan
    from $\mu_1$ to $\nu$ is induced by a map.
	For the particular case where both $\mu_1$
	and $\nu$ have compact support, this property follows from \Cref{thm:optimal_map}.
	For the general case of possibly non-compact supports,
	we refer to \cite[Theorem 10.38]{villani2009optimal}
	for a proof via approximation by compactly supported measures.
	To apply \cite[Theorem 10.38]{villani2009optimal},
	we remark that the ``Super'' assumption is proven as \Cref{lem:cost_super-diff};
	and the ``Twist'' assumption, i.e., $\nabla_y [h(d_{\matheuvm{g}}(x, y))]$ uniquely
	determines $x \in M \setminus \operatorname{Cut}(y)$, is applied and proved
	in \Cref{lem:barycenter_injectivity}.
\end{proof}

As an application of \Cref{lem:barycenter_injectivity},
\Cref{prop:barycenter_construction} and \Cref{thm:barycenter_uniqueness},
we exhibit the following structure and also the uniqueness 
\cite[Theorem 6.1]{kim2015multi} of multi-marginal optimal transport plans.

\begin{coro}
	\label{coro:multi_marginal_plan_structure}
	Let $(M, \matheuvm{g})$ be a complete Riemannian manifold.
	Let $\mu_1, \ldots, \mu_n \in \mathcal{P}(M)$ be $n \ge 2$ compactly supported measures
    with positive weights $\lambda_i > 0$ satisfying $\sum_{i=1}^n \lambda_i = 1$.
    Define $\mathbb{P} = \sum_{i=1}^n \lambda_i \, \delta_{\mu_i}$ and
    let $\mu_\mathbb{P}: = B_{\#} \gamma$ be a barycenter of $\mathbb{P}$ constructed in \Cref{prop:barycenter_construction}.
    There exists a unique continuous map
    $\matheuvm{F}: \operatorname{supp}(\mu_\mathbb{P}) \rightarrow M^n$
    such that $z \in \operatorname{bary}(\{ \matheuvm{F} (z) \})$
    for $z \in \operatorname{supp}(\mu_\mathbb{P})$,
    $\matheuvm{F} \circ B (\matheuvm x) = \matheuvm {x}$ for
    $\matheuvm {x} \in \operatorname{supp}(\gamma)$ and ${\matheuvm {F}}_{\#} \mu_{\mathbb{P}} = \gamma$.

    In particular, if $\mu_1$ is absolutely continuous, then
    $\gamma \in \Pi(\mu_1, \ldots, \mu_n)$ is the unique
    solution to the multi-marginal optimal transport problem \eqref{eq:mmot_problem}.
\end{coro}

\begin{proof}
    Since all measures $\mu_i$ for $ 1 \le i \le n$ have compact support,
    $\operatorname{supp}(\gamma)$ is compact.
    According to Statement 4 in \Cref{lem:barycenter_properties},
    the set $\operatorname{bary}(\operatorname{supp}(\gamma))$ is compact.
    It follows from $\mu_\mathbb{P} = B_{\#} \gamma$ that
    $\mu_\mathbb{P}(\operatorname{bary}(\operatorname{supp}(\gamma))) = 1$
    and thus $\operatorname{supp}(\mu_\mathbb{P}) \subset \operatorname{bary}(\operatorname{supp}(\gamma))$.
    Thanks to the injectivity of barycenter configurations in $\operatorname{supp}(\gamma)$
    (\Cref{lem:barycenter_injectivity}), 
    for any $z \in \operatorname{supp}(\mu_\mathbb{P})$,
    there exists exactly one barycenter configuration $\matheuvm{x} \in \operatorname{supp}(\gamma)$
    such that $z \in \operatorname{bary}(\{ \matheuvm{x} \})$,
    for which $\matheuvm {F}(z) := \matheuvm {x}$ is defined uniquely.
    Hence, the equality $\matheuvm{F} \circ B (\matheuvm x) = \matheuvm {x}$
    holds for any $\matheuvm {x} \in \operatorname{supp}(\gamma)$
    and any measurable barycenter selection map $B$.
    For the continuity of $\matheuvm {F}$, observe that
    for any closed subset $\matheuvm{A} \subset \operatorname{supp}(\gamma)$,
    \Cref{lem:barycenter_properties} implies that the set
    \[
        \matheuvm{F}^{-1} (\matheuvm {A})
        = \operatorname{supp}(\mu_{\mathbb{P}}) \cap \operatorname{bary}(\matheuvm{A})
    \]
    is closed (indeed compact) since $\matheuvm {A}$ is compact.
    Therefore, $\matheuvm{F}$ is measurable and 
    ${\matheuvm {F}}_{\#} \mu_{\mathbb{P}} = {[\matheuvm{F} \circ B]}_{\#} \gamma = \gamma$.

    Suppose now that $\mu_1$ is absolutely continuous.
    Given any two solutions $\gamma_1, \gamma_2 \in \Pi(\mu_1, \ldots, \mu_n)$
    of \eqref{eq:mmot_problem}, the measure
    $\gamma = \frac{1}{2} \gamma_1 + \frac{1}{2} \gamma_2 \in \Pi(\mu_1, \ldots, \mu_n) $
    is also a valid solution of \eqref{eq:mmot_problem},
    satisfying $\operatorname{supp}(\gamma_1) \cup \operatorname{supp}(\gamma_2) \subset \operatorname{supp}(\gamma)$.
    Since $\mu_\mathbb{P}$ is unique by \Cref{thm:barycenter_uniqueness},
    we must have $\mu_\mathbb{P} = B_{\#} \gamma_1 = B_{\#} \gamma_2 = B_{\#} \gamma$.
    We construct $\matheuvm {F}$ for $\gamma$ as above,
    which in particular satisfies $\matheuvm{F} \circ B (\matheuvm{x}) = \matheuvm {x}$
    for any $\matheuvm{x} \in \operatorname{supp}(\gamma_1) \cup \operatorname{supp}(\gamma_2)$.
    The previous arguments imply that
    $\gamma_1 = \gamma_2 = \matheuvm{F}_{\#} \mu_{\mathbb{P}}$,
    which is the claimed uniqueness to prove.
\end{proof}

Consistency of Wasserstein barycenters, established in \cite{le2017existence}, is a powerful
tool to prove general barycenter properties via approximation, as manifested in \cite[Proposition 5.5]{ma2025absolute}.
Confined to our settings of finitely many measures,
it is sufficient to demonstrate its following particular version, as done in \cite[Lemma 5.4]{kim2017wasserstein}.

\begin{prop}[Consistency of generalized Wasserstein barycenters]
	\label{prop:consistency_finite_case}
	Let $(M, \matheuvm{g})$ be a complete Riemannian manifold.
	Let $\mu_1, \ldots, \mu_n \in \mathcal{P}(M)$ be $n \ge 2$ probability measures
    with positive weights $\lambda_i > 0$ satisfying $\sum_{i=1}^n \lambda_i = 1$.
	For each $1 \le i \le n$, let $\{\mu_i^j\}_{j \ge 1} \subset \mathcal{P}(M)$ be a sequence of measures such that
    the weak convergence $ \mu_i^j \rightharpoonup \mu_i$ holds.
	Assume that there exists a common compact set $K \subset M$ such that $\operatorname{supp}(\mu_i) \subset K$ and $\operatorname{supp}(\mu_i^j) \subset K$ for all $1 \le i \le n$ and $j \ge 1$.
	For each $j \ge 1$, define $\mathbb{P}^j : = \sum_{i = 1}^n \lambda_i\, \delta_{\mu_i^j}$ and let $\mu_{\mathbb{P}^j}$ be a barycenter of $\mathbb{P}^j$ constructed via \Cref{prop:barycenter_construction}.
	Then there exists a sub-sequence 
    $\{\mu_{\mathbb{P}^{j_k}}\}_{k \ge 1}$ of $\{\mu_{\mathbb{P}^j}\}_{j \ge 1}$
    converging weakly to a barycenter $\mu_{\mathbb{P}}$ of $\mathbb{P} : = \sum_{i = 1}^n \lambda_i\, \delta_{\mu_i}$.
\end{prop}

\begin{proof}
	First, we establish uniform compact support for the sequence of barycenters.
	Since $\operatorname{supp}(\mu_i^j) \subset K$ for all $i, j$, any multi-marginal optimal transport plan $\gamma^j$ for $\mathbb{P}^j$ is supported on $K^n$.
	By \Cref{lem:barycenter_properties} (Statement 4), the image of the compact set $K^n$ under the barycenter operator, denoted by $K^\prime := \operatorname{bary}(K^n) \subset M$, is a compact set.
    Since $\mu_{\mathbb{P}^j}$ is formed by pushing forward $\gamma^j$ via a measurable barycenter selection map (\Cref{prop:barycenter_construction}), it follows that $\operatorname{supp}(\mu_{\mathbb{P}^j}) \subset K^\prime$ for all $j \ge 1$.

    The sequence $\{ \mu_{\mathbb{P}^j} \}_{j \ge 1}$ is uniformly tight \cite[Definition 8.6.1]{bogachev2007measure}
    since they are supported on a common compact set $K^\prime$.
    By Prokhorov's theorem \cite[Theorem 8.6.2]{bogachev2007measure}, there exists a sub-sequence 
    that converges weakly to some probability measure $\nu$ supported in $K^\prime$.
	Since both $K$ and $K^\prime$ are compact, the cost function $c(x, y) = h(d_{\matheuvm{g}}(x, y))$ is bounded and continuous on $K \times K^\prime$.
	Thanks to the stability of optimal transport \cite[Theorem 5.20]{villani2009optimal}, weak convergence of the marginals implies the convergence of the optimal transport plans, up to extraction of a sub-sequence. 
    Therefore, there exists a sub-sequence $\{\mu_{\mathbb{P}^{j_k}}\}_{k \ge 1}$ of $\{\mu_{\mathbb{P}^j}\}_{j \ge 1}$ such that, for each $i \in \{1, \dots, n\}$,
	\begin{equation}
		\label{equa:cost_convergence_consistency}
		\lim_{k \to \infty} \mathcal{T}_c(\mu_{\mathbb{P}^{j_k}}, \mu_i^{j_k}) = \mathcal{T}_c(\nu, \mu_i).
	\end{equation}

	We now verify that $\nu$ is a barycenter of $\mathbb{P}$.
	Let $\tilde{\mu}$ be an arbitrary barycenter of $\mathbb{P}$ constructed via \Cref{prop:barycenter_construction}.
    By the same compactness argument, we have $\operatorname{supp}(\tilde{\mu}) \subset K^\prime$.
	Since $\mu_{\mathbb{P}^{j_k}}$ is a barycenter of $\mathbb{P}^{j_k}$,
	\[
		\sum_{i=1}^n \lambda_i\, \mathcal{T}_c(\mu_{\mathbb{P}^{j_k}}, \mu_i^{j_k}) \le \sum_{i=1}^n \lambda_i\, \mathcal{T}_c(\tilde{\mu}, \mu_i^{j_k}).
	\]
	Taking the limit as $k \to \infty$, the left-hand side converges to $\sum_{i=1}^n \lambda_i\, \mathcal{T}_c(\nu, \mu_i)$ due to \eqref{equa:cost_convergence_consistency}.
    Applying the stability theorem to the right-hand side (where the first marginal $\tilde{\mu}$ is fixed and $\mu_i^{j_k} \rightharpoonup \mu_i$), it converges to $\sum_{i=1}^n \lambda_i\, \mathcal{T}_c(\tilde{\mu}, \mu_i)$.
	This yields:
	\[
		\sum_{i=1}^n \lambda_i\, \mathcal{T}_c(\nu, \mu_i) \le \sum_{i=1}^n \lambda_i\, \mathcal{T}_c(\tilde{\mu}, \mu_i).
	\]
	Since $\tilde{\mu}$ minimizes the right-hand side by definition, the inequality must be an equality, proving that $\nu$ is a barycenter of $\mathbb{P}$.
\end{proof}

\begin{rmk}[Consistency on proper metric spaces]
	It is worth noting that the proof of \Cref{prop:consistency_finite_case} does not invoke the differential structure of the Riemannian manifold. It relies exclusively on Prokhorov's theorem, the continuity of the cost function $h(d)$ on compact sets, and the topological properties of barycenters established in \Cref{lem:barycenter_properties}. Therefore, the existence, compactness, and consistency of generalized Wasserstein barycenters hold identically on any proper metric space $(E, d)$.
\end{rmk}

\section{Absolute continuity of barycenter measures}

\label{sec:absolute_continuity}

This section is devoted to the proof of the absolute continuity of generalized Wasserstein barycenters.
Our approach relies on the regularity of the optimal transport map from
the barycenter measure $\mu_\mathbb{P}$ back to the absolutely continuous marginal $\mu_1$.
Following the strategy developed in \cite{ma2025absolute}, we approximate the remaining marginals $\mu_2, \ldots, \mu_n$ by discrete measures.
This allows us to disintegrate the multi-marginal optimal plan into conditional slices where the points $\matheuvm{x}^\prime = (x_2, \ldots, x_n)$ act as fixed spatial anchors.
Given a target barycenter $z$ generated by a configuration $(x_1, \matheuvm {x}^\prime)$, the first-order condition \eqref{equa:first_order_condition} uniquely determines the source point $x_1$
thanks to \Cref{lem:barycenter_injectivity}.
We can thus express $x_1$ as an inverse map $F(z) = x_1$ defined explicitly by
\begin{equation}
	\label{eq:explicit_F_intro}
	F(z) = \exp_z \left( -(h^\prime)^{-1}(\|V(z)\|) \frac{V(z)}{\|V(z)\|} \right), \quad \text{where }
	V(z) := -\frac{1}{\lambda_1} \sum_{i=2}^n \lambda_i \nabla_z[h(d_{\matheuvm{g}}(z, x_i))].
\end{equation}

The absolute continuity of the barycenter is ultimately governed by the volume distortion of $F$.
However, for a general cost lacking quadratic growth, such as the case $h(d) = d^p / p$ with $1 < p < 2$, this map exhibits severe analytical singularities.
Specifically, when the barycenter $z$ collides with a marginal point $x_i$, the distance $d_{\matheuvm{g}}$ vanishes.
Since $h$ fails to be $\mathcal{C}^2$ at the origin, the vector field $V(y)$ loses its differentiability, and the map $F$ ceases to be locally Lipschitz continuous.

In the Euclidean setting, Brizzi, Friesecke, and Ried \cite{brizzi2025p} address this lack of regularity by analyzing the barycenter map defined on the full configuration space \cite[section 2.2]{brizzi2025p}.
They partition the support of the optimal plan into a regular set, where the barycenter is disjoint from all marginal points, and singular sets, where it coincides with one or more marginals.
On the regular set, they leverage the $c$-cyclical monotonicity \cite[Proposition 3.5]{brizzi2026h} of the optimal plan to establish a local quantitative injectivity estimate \cite[Proposition 3.2]{brizzi2025p}.
It is a local Lipschitz estimate allowing them to use the $m$-dimensional Hausdorff measure to show that the projection of this preimage onto any marginal space preserves Lebesgue null sets \cite[Lemma 3.3]{brizzi2025p}.
However, deriving this multi-dimensional injectivity estimate intrinsically relies on the strict translational symmetry of flat space to algebraically deduce \cite[(5.5)]{brizzi2026h},
when $h(d_{\matheuvm{g}}(z, x_i)) = \| z - x_i \|$,
\begin{equation}
	\label{equa:euclidean_symmetry}
	\operatorname{Hess}_{x_i} [h(d_{\matheuvm{g}}(z, x_i))] = \operatorname{Hess}_z[h(d_{\matheuvm{g}}(z, x_i))] = - \nabla_{x_i} \nabla_z[h(d_{\matheuvm{g}}(z, x_i))].
\end{equation}
On general Riemannian manifolds, this symmetry is broken by the twisting of Jacobi fields, rendering such  inverse-Lipschitz bounds intractable.
Furthermore, on the singular sets, the barycenter map acts locally as a coordinate projection \cite[(3.4)]{brizzi2025p}.
To ensure the barycenter measure assigns zero mass to Lebesgue null sets within these regions, the corresponding marginals with collisions are assumed to be absolutely continuous \cite[Theorem 3.5]{brizzi2025p}.

Instead of estimating the Hausdorff measure of projected pre-images from the full configuration space, our discrete approximation framework allows us to analyze the barycenter map directly on the manifold. By discretizing the measures $\mu_2, \ldots, \mu_n$, we freeze the coordinates $(x_2, \ldots, x_n) = \matheuvm{v}$ into fixed spatial anchors.
This anchoring strategy, adapted from \cite[section 2.3]{ma2025absolute},
allows us to completely bypass the need for the restrictive symmetry \eqref{equa:euclidean_symmetry} to control how simultaneous variations across all marginals distort the barycenter.
To explicitly emphasize this dependence on the conditional slicing, we hereafter denote the inverse map as $F_{\matheuvm{v}}(z)$.
The volume distortion problem is thus reduced entirely to bounding the differential $\di F_{\matheuvm{v}} (z)$ on conditional slices to deduce uniform Lipschitz estimates.

To ensure the Lipschitz continuity of $F_{\matheuvm{v}}$, the analytical singularities at $d=0$ must be resolved. Two natural strategies arise. The first is geometric exclusion: enforcing a strictly positive lower bound on the distances $d_{\matheuvm{g}}(z, x_i)$ inherently shields $F_{\matheuvm{v}}$ from the non-differentiability of the distance function and the origin singularity of $(h^\prime)^{-1}$.
The second strategy is analytical smoothing: rather than avoiding collisions, we assume the cost profile $h$ is sufficiently flat at the origin.
Specifically, the $\mathcal{C}^2$-continuity of $h$ at $0$, typical for $h(d) = d^p/p$ with $p \ge 2$,
\[
	\lim_{t \downarrow 0} h^{\prime\prime}(t) = h^{\prime\prime}(0)
	:= \lim_{t \downarrow 0} \frac{h^\prime(t)}{t}
\]
ensures that the Riemannian Hessian of $h(d_{\matheuvm{g}}(z, x_i))$ varies smoothly at $z = x_i$.
This neutralizes the singularity, preserving the Lipschitz continuity of $F_{\matheuvm{v}}$ directly across marginal collisions.

\subsection{Absolute continuity via approximation}
\label{sec:absolute_continuity_approximation}

In this subsection, we formalize the measure-theoretic framework that allows us to deduce the absolute continuity of a target measure from a sequence of discrete approximations.
To track absolute continuity under weak convergence,
we introduce the following characterizing class defined for positive real numbers $\epsilon, \delta > 0$,
\begin{equation}
	\label{eq:defn_E_epsilon_delta}
	\mathcal{E}_{\epsilon, \delta} := \left\{ \mu \in \mathcal{P}(M) \;\middle|\; \forall\, N \in \mathcal{B}(M),\, \operatorname{Vol}(N) < \delta \implies \mu(N) \le \epsilon \right\}.
\end{equation}
As established in \cite[Lemma 2.12]{ma2025absolute}, the set of all absolutely continuous probability measures on $M$ can be fully characterized as
\[
	\mathcal{A} = \bigcap_{k \in \mathbb{N}} \bigcup_{l \in \mathbb{N}} \mathcal{E}_{2^{-k}, 2^{-l}}.
\]
Crucially, each class $\mathcal{E}_{\epsilon, \delta}$ is shown to be
closed with respect to the weak convergence of measures.
This topological closure ensures that if a sequence of approximating measures belongs to a common $\mathcal{E}_{\epsilon, \delta}$ class, their weak limit inherently retains the same volume-distortion bounds.

We now synthesize this property into a general proposition, based on the proofs of \cite[section 2.3]{ma2025absolute}.
To accommodate the decomposition strategy employed in our main theorem, we state the result
with slightly relaxed assumptions.

\begin{prop}
	\label{prop:absolute_continuity_via_approximation}
	Let $(M, \matheuvm{g})$ be a complete Riemannian manifold.
	Fix $n \ge 2$ compactly supported
	probability measures $\mu_1, \ldots, \mu_n \in \mathcal{P}(M)$ with positive weights $\lambda_i > 0$
	satisfying $\sum_{i=1}^n \lambda_i = 1$.
	Suppose that $\mu_1$ is absolutely continuous with respect to $\operatorname{Vol}$.
	Define $\mathbb{P} := \sum_{i=1}^n \lambda_i\, \delta_{\mu_i}$.
	According to \Cref{prop:barycenter_construction} and \Cref{thm:barycenter_uniqueness},
	the unique barycenter of $\mathbb{P}$ is $\mu_{\mathbb{P}} := B_{\#} \gamma$, where $\gamma \in \Pi(\mu_1, \ldots, \mu_n)$
	is a multi-marginal optimal transport plan and
	$B: M^n \to M$ is a measurable barycenter selection map.
	Introduce the closed set $\Gamma = \{ (\matheuvm{x}, z) \mid z \in \operatorname{bary}(\{\matheuvm{x}\}) \}$
	as in \Cref{lem:barycenter_properties}.
	The barycenter $\mu_\mathbb{P}$ is absolutely continuous provided that
	there exist a Borel set $\Omega \subset \Gamma$ and a constant $L > 0$
	satisfying the following two conditions:
	\begin{enumerate}
		\item $\widetilde \gamma \left(\widebar{\Gamma \setminus \Omega }\right) = 0$ with $\widetilde \gamma : = (\operatorname{Id}_{M^n}, B)_{\#} \gamma$;
		\item for any point $\matheuvm {v} := (x_2, \ldots, x_n) \in M^{n-1}$
		      such that $\Omega$ intersects with $ M \times \{ \matheuvm{v} \} \times M $,
		      there exists a Borel set $Z \in \mathcal{B}(M)$ and an $L$-Lipschitz continuous
		      function $F_{\matheuvm{v}}: Z \rightarrow M$ satisfying
		      \[
			      \{ (F_{\matheuvm{v}}(z), \matheuvm{v}, z) \mid z \in Z \}
			      = [ M \times \{ \matheuvm{v} \} \times M ] \cap \Omega.
		      \]
	\end{enumerate}
\end{prop}

\begin{proof}
	For each $\mu_i$ with $2 \le i \le n$, we approximate it in the weak topology with finitely supported
	measures $\{ \mu_i^j \}_{ j \ge 1 }$ such that $\operatorname{supp}(\mu_i^j) \subset \operatorname{supp}(\mu_i)$.
	Denote by $\gamma^j \in \Pi(\mu_1, \mu_2^j, \ldots, \mu_n^j)$ the multi-marginal optimal
	transport for the measure
	$\mathbb{P}^j = \lambda_1 \delta_{\mu_1} + \sum_{i=2}^n \lambda_i \delta_{\mu_i^j}$,
	which is unique according to \Cref{coro:multi_marginal_plan_structure}
	since $\mu_1$ is absolutely continuous.
	By \Cref{prop:barycenter_construction} and \Cref{thm:barycenter_uniqueness},
	$\mu_{\mathbb{P}^j} : = B_{\#} \gamma^j$ is the unique barycenter of $\mathbb{P}^j$.
	Up to extracting a sub-sequence, we can assume without loss of generality that
	the weak convergence $\gamma^j \rightharpoonup \gamma$ by the stability
	of optimal transport plans \cite[Theorem 5.20]{villani2009optimal},
	and the weak convergence $\mu_{\mathbb{P}^{j}} \rightharpoonup \mu_{\mathbb{P}}$
	by \Cref{prop:consistency_finite_case}.
	Hence, the sequence $\widetilde \gamma^j : = (\operatorname{Id}_{M^n}, B)_{\#} \gamma^j$
	is tight in $\mathcal{P}(\Gamma)$ because $\operatorname{supp}(\widetilde \gamma^j) \subset \Gamma$
	and the marginals $\gamma^j$ and $\mu_{\mathbb{P}^j}$ converge weakly.
	We prove $\widetilde \gamma^j$ converges weakly to $\widetilde \gamma$
	by showing that any possible weak limit $\eta \in \mathcal{P}(\Gamma)$
	of $\{\widetilde \gamma^j\}_{j \ge 1}$ after extracting a sub-sequence must coincide with $\widetilde \gamma$.
	Fix a point $(\matheuvm{x}, z) \in \operatorname{supp}(\eta)$.
	By the characterization of weak convergence via continuous functions,
	$\gamma$ and $\mu_{\mathbb{P}} = B_{\#} \gamma$ in this order are
	the two marginals of $\eta$ since projection maps are continuous,
	which implies $\matheuvm{x} \in \operatorname{supp}(\gamma)$
	and $z \in \operatorname{supp}(\mu_\mathbb{P})$.
	Moreover, since $\operatorname{supp}(\eta) \subset \Gamma$,
	$z \in \operatorname{bary}(\{ \matheuvm {x} \})$,
	\Cref{lem:barycenter_injectivity} implies that $\matheuvm {x}$
	is uniquely determined by $z$ via the map
	$\matheuvm {F}(z) = \matheuvm {x}$ in \Cref{coro:multi_marginal_plan_structure}.
	Hence, $\eta$ assigns full mass to
	the image of $\operatorname{supp}(\mu_\mathbb{P})$ under the map $z \mapsto (\matheuvm{F} (z), z)$.
	It follows that
	\begin{equation}
		\label{equa:lift_measure_as_push_forward}
		\eta = (\matheuvm{F}, \operatorname{Id}_M)_{\#} \mu_\mathbb{P}
		= (\matheuvm{F} \circ B, B)_{\#} \gamma
		= (\operatorname{Id}_{M^n}, B)_{\#} \gamma
		= \widetilde \gamma.
	\end{equation}

	Define $\matheuvm {K} : = \widebar{\Gamma \setminus \Omega }$
	and $\matheuvm {O}: = M^{n+1} \setminus \matheuvm{K}$.
	Since $\widetilde \gamma^j \rightharpoonup \widetilde \gamma$
	and $\matheuvm {K}$ is closed with $\widetilde \gamma(\matheuvm {K}) = 0$ by Condition 1,
	$\lim_{j \rightarrow \infty} \widetilde \gamma^j (\matheuvm {K}) = \widetilde {\gamma}(\matheuvm {K}) = 0$.
	We first prove the following claim, for any $\epsilon, \delta >0$
	and the set $\mathcal{E}_{\epsilon, \delta}$ defined in \eqref{eq:defn_E_epsilon_delta},
	\begin{equation}
		\label{equa:ac_sets_implication}
		\mu_1 \in \mathcal{E}_{\epsilon, \delta} \implies
		\mu_{\mathbb{P}} \in \mathcal{E}_{\epsilon, \delta / L^m}.
	\end{equation}
	Fix $\epsilon, \delta >0$ and suppose $\mu_1 \in \mathcal{E}_{\epsilon, \delta}$.
	Let $N \in \mathcal{B}(M)$ be an arbitrarily chosen Borel set
	such that $\operatorname{Vol} (N) < \delta / L^m$.
	Since the volume measure $\operatorname{Vol}$ is outer regular,
	there exists an open neighborhood $U$ of $N$ such that
	$\operatorname{Vol}(N) \le \operatorname{Vol}(U) < \delta / L^m$.
	By the weak convergence $\mu_{\mathbb{P}^j} \rightharpoonup \mu_{\mathbb{P}}$,
	\begin{align}
		\label{equa:pass_to_limit_barycenter_measure}
		\mu_{\mathbb{P}}(N) \le \mu_{\mathbb{P}}(U)
		\le \liminf_{j \rightarrow \infty}\mu_{\mathbb{P}^j}(U)
		 & \le \liminf_{j \rightarrow \infty} \left\{\widetilde \gamma^j([M^n \times U] \cap \matheuvm {O})
		\nonumber
		+ \widetilde \gamma^j(\matheuvm {K})\right\}                                                        \\
		 & = \liminf_{j \rightarrow \infty} \widetilde \gamma^j |_{\matheuvm{O}} (M^n \times U).
	\end{align}
	Hence, to prove claim \eqref{equa:ac_sets_implication}, it suffices to show
	that $\widetilde \gamma^j|_{\matheuvm{O}}(M^n \times U) \le \epsilon$,
	which is a slightly generalized version of \cite[Proposition 2.22]{ma2025absolute}.
	We fix an index $j$ and reformulate its proof for completeness.
	Applying \Cref{coro:multi_marginal_plan_structure}
	to $\widetilde \gamma^j$,
	we obtain two continuous maps $T_1: \operatorname{supp}(\mu_{\mathbb{P}^j}) \rightarrow M$ and $T_2: \operatorname{supp}(\mu_{\mathbb{P}^j}) \rightarrow M^{n-1}$,
	such that
	\[
		\widetilde \gamma^j = (T_1, T_2, \operatorname{Id}_M)_{\#} \mu_{\mathbb{P}^j}.
	\]
	Since $\mu_i^j$ for $2 \le i \le n$ are finitely supported measures,
	$T_2$ takes at most finitely many distinct values.
	Hence, the inclusion $\Gamma \cap \matheuvm{O} \subset \Omega$ and Condition 2 imply
	that there are Borel sets $Y_1, \ldots, Y_l \subset \operatorname{supp}(\mu_{\mathbb{P}^j})$
	and distinct values $\matheuvm{v}_1, \ldots, \matheuvm{v}_l \in M^{n-1}$
	satisfying $T_2 |_{Y_k} \equiv \matheuvm{v}_k$,
	\begin{equation}
		\label{equa:support_Lipschitz_inclusion}
		\operatorname{supp}(\widetilde \gamma^j) \cap {\matheuvm{O}}
		\subset \bigcup_{k=1}^{l}
		\{ (F_{\matheuvm{v}_k}(z) , \matheuvm{v}_k, z) \mid z \in Y_k \}
		\subset \operatorname{supp}(\widetilde \gamma^j),
	\end{equation}
	and $F_{\matheuvm{v}_k}$ is $L$-Lipschitz continuous on $Y_k$.
	Note that $F_{\matheuvm{v}_k} |_{Y_k} = T_1 |_{Y_k}$
	by our construction of $(T_1, T_2):  \operatorname{supp}(\mu_{\mathbb{P}^j}) \rightarrow  M^n$
	using \Cref{coro:multi_marginal_plan_structure},
	and $Y_1, \ldots, Y_l$ are disjoint since their images
	under $T_2$ are distinct values $\matheuvm{v}_1, \ldots, \matheuvm{v}_l$.
	Define $Y : = \cup_{k=1}^l Y_k$.
	The $L$-Lipschitz continuity of $T_1$ on its disjoint components $Y_k$ implies
	\cite[Proposition 12.6, Proposition 12.12, Remark after Proposition 12.12]{taylor2006measure}
	\[
		\operatorname{Vol}(T_1(Y \cap U)) \le
		\sum_{k=1}^l \operatorname{Vol}(T_1(Y_k \cap U)) \le
		\sum_{k=1}^l L^m \operatorname{Vol}(Y_k \cap U) =
		L^m \operatorname{Vol}(U)
		< \delta.
	\]
	Though $T_1(Y \cap U)$ is not necessarily a Borel set,
	there exists a Borel set $W \in \mathcal{B}(M)$ containing it
	with equal volume measure \cite[Corollary 1.5.8]{bogachev2007measure}.
	Since $\mu_1 \in \mathcal{E}_{\epsilon, \delta}$ and $T_1(Y \cap U) \subset W$,
	\eqref{equa:support_Lipschitz_inclusion} implies
	\[
		\widetilde{\gamma}^j|_{\matheuvm{O}}(M^n \times U)
		\le \widetilde{\gamma}^j|_{\matheuvm{O}}(W \times M^n)
		\le \widetilde{\gamma}^j (W \times M^n)
		= \mu_1(W) \le \epsilon,
	\]
	since $\operatorname{Vol}(W) = \operatorname{Vol}(T_1(Y \cap U)) < \delta$
	and $\mu_1 \in \mathcal{E}_{\epsilon, \delta}$ by assumption.
	This concludes the proof of our claim \eqref{equa:ac_sets_implication},
	thanks to the inequality \eqref{equa:pass_to_limit_barycenter_measure}.

	Since $\mu_1$ is absolutely continuous,
	\eqref{equa:ac_sets_implication} implies that
	$\mu_\mathbb{P}$ is also absolutely continuous \cite[Lemma 2.12]{ma2025absolute}.
\end{proof}

\begin{rmk}
	For the quadratic case $h(d) = d^2 / 2$, $F_{\matheuvm{v}}(y)$ is reduced to $\exp_y(-V_{\matheuvm{v}}(y))$
	and the proof of \cite[Theorem 2.24]{ma2025absolute} shows that the set
	\[
		\Omega := \Gamma \cap [\operatorname{supp}(\mu_1) \times \cdots \times \operatorname{supp}(\mu_n) \times M]
	\]
	satisfies the two required conditions, with $L$ determined by the compact supports and
	the Rauch comparison theorem.
\end{rmk}

\subsection{Lipschitz continuity on regular regions}
\label{sec:lipschitz_regular_regions}

Before proving the Lipschitz continuity required by Condition 2 of \Cref{prop:absolute_continuity_via_approximation}, we first address the explicit construction of the inverse map $F_{\matheuvm{v}}$.

\begin{lem}[Explicit expression of the inverse map]
	\label{lem:explicit_inverse_map}
	Let $(M, \matheuvm{g})$ be a complete Riemannian manifold.
	Fix $n \ge 2$ positive coefficients $\lambda_i > 0$ satisfying $\sum_{i=1}^n \lambda_i = 1$, and define the closed set $\Gamma = \{ (\matheuvm{x}, z) \mid z \in \operatorname{bary}(\{\matheuvm{x}\}) \}$
	as in \Cref{lem:barycenter_properties}.
	For any point $(x_1, \matheuvm{v}, z) \in \Gamma$ with $\matheuvm{v} = (x_2, \ldots, x_n) \in M^{n-1}$, if $z \neq x_1$, then $x_1$ is uniquely determined by $z$ and $\matheuvm{v}$ via the map $x_1 = F_{\matheuvm{v}}(z)$, where
	\begin{equation}
		\label{equa:F_v_explicit_expression}
		F_{\matheuvm{v}}(z) :=
		\exp_z \left( -(h^\prime)^{-1}(\|V_{\matheuvm{v}}(z)\|) \frac{V_{\matheuvm{v}}(z)}{\|V_{\matheuvm{v}}(z)\|} \right), \quad \text{and} \quad
		V_{\matheuvm{v}}(z) :=
		-\frac{1}{\lambda_1} \sum_{i=2}^n \lambda_i \nabla_z[h(d_{\matheuvm{g}}(z, x_i))].
	\end{equation}
\end{lem}

\begin{proof}
	Since $(\matheuvm{x}, z) \in \Gamma$, $z$ is a barycenter of the configuration $(x_1, \matheuvm{v})$.
	By \Cref{lem:cut_locus_avoidance}, $z \notin \operatorname{Cut}(x_i)$ for any $1 \le i \le n$.
	Hence, the first-order optimality condition \eqref{equa:first_order_condition} holds, yielding
	\[
		\lambda_1 \nabla_z[h(d_{\matheuvm{g}}(z, x_1))] + \sum_{i=2}^n \lambda_i \nabla_z [h(d_{\matheuvm{g}}(z, x_i))] = 0.
	\]
	As $z \neq x_1$, we have $d_{\matheuvm{g}}(z, x_1) > 0$.
	Rearranging the above equation gives $V_{\matheuvm{v}}(z) =
		\nabla_z[h(d_{\matheuvm{g}}(z, x_1))] =
		h^\prime(d_{\matheuvm{g}}(z, x_1))
		\nabla_z d_{\matheuvm{g}}(z, x_1)$.
	Since $\| \nabla_z d_{\matheuvm{g}}(z, x_1) \| = 1$,
	$ \| V_{\matheuvm{v}}(z) \| =  h^\prime(d_{\matheuvm{g}}(z, x_1))$.
	Our lemma thus follows from the same arguments in the proof of \Cref{lem:tangency}.
\end{proof}

The most straightforward way to ensure that $F_{\matheuvm{v}}$ is uniformly Lipschitz is to bound the configuration away from all collisions. This avoids both the singularity of $(h^\prime)^{-1}$ and the non-differentiability of $V_{\matheuvm{v}}$.

\begin{lem}[Lipschitz continuity strictly away from collisions]
	\label{lem:lipschitz_no_collision}
	Under the settings of \Cref{lem:explicit_inverse_map}, for any $\alpha > 0$, define the collision-free subset
	\begin{equation}
		\Omega_{\alpha} := \left\{ (\matheuvm{x}, z) \in \Gamma \;\middle|\; d_{\matheuvm{g}}(z, x_i) > \alpha \text{ for } 1 \le i \le n \right\}.
	\end{equation}
	For any compact subset $K \subset \Gamma$, there exists a constant $L > 0$ (depending on $K$ and $\alpha$) such that
	Condition $2$ of \Cref{prop:absolute_continuity_via_approximation} is satisfied by $\Omega_{\alpha} \cap K$.
	In particular, the map $F_{\matheuvm{v}}: Z \to M$ defined by \eqref{equa:F_v_explicit_expression}
	is $L$-Lipschitz continuous, where $Z := p_2(\left[M \times \{ \matheuvm{v} \} \times M \right] \cap \Omega_{\alpha} \cap K)$ and $p_2: M^{n} \times M \rightarrow M$ is the canonical projection.
\end{lem}

\begin{proof}
	To uniformly bound the Lipschitz constant of $F_{\matheuvm{v}}$ independently of $\matheuvm{v}$, we treat the coordinates $\matheuvm{v} = (x_2, \ldots, x_n)$ and $z$ jointly. Let $\mathcal{D} \subset M^{n-1} \times M$ be the projection of the compact set $\Omega_\alpha \cap K$ onto the variables $(\matheuvm{v}, z)$.
	Since $K$ is compact, $\mathcal{D}$ is a pre-compact set.

	We define the joint function $f: M^{n-1} \times M \to \mathbb{R}$ and the map $W: TM \setminus \{0\} \to TM$ respectively by
	\[
		f(\matheuvm{v}, z) = -\frac{1}{\lambda_1} \sum_{i=2}^n \lambda_i \, h(d_{\matheuvm{g}}(z, x_i))
		\quad \text{and} \quad
		W(V) = -(h^{\prime})^{-1}(\|V\|) \frac{V}{\| V \|}.
	\]
	By \Cref{lem:cut_locus_avoidance}, for any point $(\matheuvm{v}, z) \in \mathcal{D}$, the barycenter $z$ is disjoint from the cut loci of all $x_i$.
	Hence, there exists a pre-compact open neighborhood $\mathcal{U} \subset M^{n-1} \times M$ of $\widebar {\mathcal{D}}$ such that for any $(\matheuvm{v}, y) \in \widebar{\mathcal{U}}$, $y \notin \operatorname{Cut}(x_i)$ and $d_{\matheuvm{g}}(y, x_i) > \alpha / 2$ for $2 \le i \le n$.
	As $h$ is $\mathcal{C}^2$-smooth on $(0, + \infty)$ and the intrinsic distance functions are smooth away from the cut locus, $f$ is a $\mathcal{C}^2$ function on $\mathcal{U}$. Thus, the partial gradient $\nabla_y f(\matheuvm{v}, y)$ and the partial Hessian $\operatorname{Hess}_y f(\matheuvm{v}, y)$ exist and are continuous on $\mathcal{U}$.

	According to \eqref{equa:F_v_explicit_expression}, we can jointly express the inverse map as $\mathcal{G}(\matheuvm{v}, y) := \exp_y \circ W (\nabla_y f(\matheuvm{v}, y))$.
	As shown in the proof of \Cref{lem:explicit_inverse_map}, for points arising from $\Omega_\alpha \cap K$, we have $\| \nabla_z f(\matheuvm{v}, z) \| = h^\prime(d_{\matheuvm{g}}(z, x_1))$.
	By the condition $d_{\matheuvm{g}}(z, x_1) > \alpha$ in the definition of $\Omega_\alpha$ and the strict convexity of $h$ \ref{assumption:h_strict_convexity}, we have $\| \nabla_z f(\matheuvm{v}, z) \| \ge h^\prime(\alpha) > 0$ on $\mathcal{D}$.
	By shrinking the neighborhood $\mathcal{U}$ if necessary, we can assume without loss of generality that $\| \nabla_y f(\matheuvm{v}, y) \| > 0$ holds for all $(\matheuvm{v}, y) \in \mathcal{U}$, ensuring that the composition $W(\nabla_y f(\matheuvm{v}, y))$ is well-defined.

	We recall that, even for the quadratic case $h(d) = d^2 / 2$, the partial differential $\di_y \mathcal{G}$ is related to $\operatorname{Hess}_y f$ via a relatively complicated formula \cite[Lemma 1.43]{ma2025thesis}.
	For simplicity, we now show that $\di_y \mathcal{G}$ exists and is continuous on $\mathcal{U}$ without deducing its explicit formula.
	Since $\operatorname{Hess}_y f$ is the covariant derivative of the partial differential
	$\di_y f$ (metric dual of $\nabla_y f$) \cite[Proposition 2.2.6]{petersen2016riemannian}, the tangent bundle perspective of this covariant derivative \cite[Proposition 4.1]{sakai1996riemannian} implies that $(\matheuvm{v}, y) \mapsto \nabla_y f(\matheuvm{v}, y)$ is a $\mathcal{C}^1$ map from $\mathcal{U}$ into $TM$.
	As the exponential map and $W$ are both smooth maps on their domains, the composition $\mathcal{G}: \mathcal{U} \rightarrow M$ is a $\mathcal{C}^1$ map.

	Since $\mathcal{G}$ is a $\mathcal{C}^1$ map, its partial differential $\di_y \mathcal{G}(\matheuvm{v}, y)$ is continuous on $\mathcal{U}$.
	As $\widebar {\mathcal{D}} \subset \mathcal{U}$ is compact, the operator norm $\|\di_y \mathcal{G}(\matheuvm{v}, z)\|$ admits a finite maximum value $L < \infty$ over all $(\matheuvm{v}, z) \in \mathcal{D}$. It follows that the sliced map $\mathcal{G}(\matheuvm{v}, \cdot) = F_{\matheuvm{v}}$ is an $L$-Lipschitz continuous map on its respective domain $Z$, with $L$ independent of $\matheuvm{v}$.

	The set $\Omega: = \Omega_\alpha \cap K$ is a Borel set as $\Omega_\alpha$ is a relatively
	open subset of the closed set $\Gamma$.
	To show that $Z$ is also a Borel set,
	we extend the definition domain of $F_{\matheuvm{v}}$ to
	an open neighborhood $U$ of $Z$
	such that $\{\matheuvm{v}\} \times U \subset \mathcal{U}$,
	by setting $F_{\matheuvm{v}}(y) := \mathcal{G}(\matheuvm{v}, y)$ for $y \in U$.
	Since $F_{\matheuvm{v}}: U \rightarrow M$ is $\mathcal{C}^1$,
	the map $y \mapsto (F_{\matheuvm{v}}(y), \matheuvm{v}, y)$
	defined on $U$ is measurable, and its pre-image of $\Omega$,
	being exactly the set $Z$, is thus a Borel set.
\end{proof}

While $\Omega_\alpha$ guarantees Lipschitz continuity, restricting ourselves to strictly collision-free regions requires all marginal measures to be absolutely continuous to handle the discarded collision sets.
To relax this requirement, we now determine if $F_{\matheuvm{v}}(z)$ remains Lipschitz even when $z$ collides with $x_i$ for $i \ge 2$.
Fix a point $x_i$, define $r(z) := d_{\matheuvm{g}}(z, x_i)$ and $f_i(z) := h(d_{\matheuvm{g}}(z, x_i))$.
The analytical hurdle lies in the Hessian of $f_i$ when $z \rightarrow x_i$.
Applying the chain rule, we obtain \cite[Exercise 2.5.6]{petersen2016riemannian}
\begin{equation}
	\label{eq:hessian_expansion}
	\operatorname{Hess}_z
	f_i(z) = h^{\prime\prime}(r) \diff r \otimes \diff r + \frac{h^\prime(r)}{r}
	\left( r\, \operatorname{Hess}_z r \right),
\end{equation}
where $\diff r $ denotes the differential of the function $r$ (i.e., $\di r$)
that is only well-defined for $z \neq x_i$.
Moreover, the tensor product $\diff r \otimes \diff r$
is usually written as symmetric product $\diff r \diff r$ or $\diff r^2$, while we keep current notation for clarity.
By definition,
for tangent vectors $X, Y \in T_zM$,
$\diff r \otimes \diff r (X, Y)$ is the product of scalars $X r $ and $Y r$.
Thanks to the classical result \cite[section 5.5.3]{petersen2016riemannian},
\begin{equation}
	\label{equa:expansion_hessian_r_at_0}
	\lim_{z \rightarrow x_i} \left(\operatorname{Hess}_z r(z) - \frac{\matheuvm{g} - \diff r \otimes \diff r}{r(z)} \right)
	= 0,
\end{equation}
we deduce the following regularity result.

\begin{lem}[Lipschitz continuity allowing marginal collisions]
	\label{lem:lipschitz_with_collision}
	Under the settings of \Cref{lem:explicit_inverse_map}, suppose the cost function satisfies the following local condition at the origin:
	\begin{equation}
		h^{\prime\prime}(0) := \lim_{t \downarrow 0} \frac{h^\prime(t)}{t} = \lim_{t \downarrow 0} h^{\prime\prime}(t).
	\end{equation}
	For any $\alpha > 0$, define the partially relaxed subset
	\begin{equation}
		\Omega_{1, \alpha} := \left\{ (\matheuvm{x}, z) \in \Gamma \;\middle|\; d_{\matheuvm{g}}(z, x_1) > \alpha \right\}.
	\end{equation}
	Then, for any compact subset $K \subset \Gamma$, there exists a constant $L > 0$ such that Condition $2$ of \Cref{prop:absolute_continuity_via_approximation} is satisfied by $\Omega_{1, \alpha} \cap K$.
\end{lem}

\begin{proof}
	Let $\mathcal{Q} \subset M^{n-1} \times M$ be the projection of $\Omega_{1, \alpha} \cap K$ onto the variables $(\matheuvm{v}, z)$.
	The set $\mathcal{Q}$ is pre-compact as $K$ is compact.
	Following the framework established in the proof of \Cref{lem:lipschitz_no_collision}, we define the joint function $f(\matheuvm{v}, y) = -\frac{1}{\lambda_1} \sum_{i=2}^n \lambda_i\,h(d_{\matheuvm{g}}(y, x_i))$.
	By \Cref{lem:cut_locus_avoidance}, no barycenter configuration involves the cut locus.
	Thus, there exists a pre-compact open neighborhood $\mathcal{U} \subset M^{n-1} \times M$ of the compact closure $\widebar{\mathcal{Q}}$ such that $y \notin \operatorname{Cut}(x_i)$
	for $2 \le i \le n$ and any $(\matheuvm{v}, y) \in \widebar{\mathcal{U}}$.

	For any point $(\matheuvm{v}, z) \in \mathcal{Q}$, the first-order optimality condition and the bound $d_{\matheuvm{g}}(z, x_1) > \alpha$ guarantee that $\| \nabla_z f(\matheuvm{v}, z) \| = h^\prime(d_{\matheuvm{g}}(z, x_1)) \ge h^\prime(\alpha) > 0$.
	By the continuity of the gradient, we can shrink the neighborhood $\mathcal{U}$ if necessary to ensure that the strict lower bound $\| \nabla_y f(\matheuvm{v}, y) \| > h^\prime(\alpha)/2 > 0$ holds uniformly for all $(\matheuvm{v}, y) \in \mathcal{U}$.
	By the exact same composition argument used in \Cref{lem:lipschitz_no_collision}, to establish the uniform Lipschitz continuity of $F_{\matheuvm{v}}$, it suffices to show that the partial gradient map $\nabla_y f: \mathcal{U} \rightarrow TM$ is a $\mathcal{C}^1$ map between manifolds.
	Thanks to the tangent bundle perspective of the covariant derivative \cite[Proposition 4.1]{sakai1996riemannian}, this reduces to proving that the Riemannian Hessian of each term $f_i(y) := h(d_{\matheuvm{g}}(y, x_i))$ is well-defined and continuous everywhere on $\mathcal{U}$, including at the collision $y = x_i$.

	Fix an index $i \ge 2$. Since $h^\prime(0) = 0$, the gradient $\nabla_y f_i(y) = h^\prime(d_{\matheuvm{g}}(y, x_i)) \nabla_y d_{\matheuvm{g}}(y, x_i)$ vanishes continuously as $y \rightarrow x_i$, meaning $x_i$ is a critical point of $f_i$.
	At a critical point, the Riemannian Hessian can be evaluated as the standard second derivative \cite[section 2.4]{petersen2016riemannian} along any unit-speed geodesic $\gamma(t)$ starting at $\gamma(0) = x_i$.
	Since $d_{\matheuvm{g}}(\gamma(t), x_i) = t$ for small $t > 0$, we have $f_i(\gamma(t)) = h(t)$.
	The second derivative at $t=0$ yields $\operatorname{Hess}_{y} f_i|_{y = x_i}(\dot{\gamma}(0), \dot{\gamma}(0)) = h^{\prime\prime}(0)$.
	Since this holds for all directions, we explicitly obtain $\operatorname{Hess}_{y} f_i|_{y = x_i} = h^{\prime\prime}(0) \matheuvm{g}$.

	It remains to show that the Hessian is continuous as $y \rightarrow x_i$. Let $r(y) := d_{\matheuvm{g}}(y, x_i)$.
	For $y \neq x_i$, the chain rule \eqref{eq:hessian_expansion} isolates the non-convergent tensor $\diff r \otimes \diff r$ ($\diff r$ is not well-defined at $y= x_i$):
	\[
		\operatorname{Hess}_y f_i = \frac{h^\prime(r)}{r} \matheuvm{g} + \left( h^{\prime\prime}(r) - \frac{h^\prime(r)}{r} \right) \diff r \otimes \diff r + \frac{h^\prime(r)}{r} \left( r\, \operatorname{Hess}_y r - (\matheuvm{g} - \diff r \otimes \diff r) \right).
	\]
	We evaluate the limit as $y \rightarrow x_i$ (i.e., $r \downarrow 0$).
	The third term vanishes due to the geometric asymptotic expansion \eqref{equa:expansion_hessian_r_at_0}.
	By our local origin-flatness assumption, both $h^{\prime\prime}(r)$ and $\frac{h^\prime(r)}{r}$ converge to the same value $h^{\prime\prime}(0)$, implying that the scalar coefficient $\left( h^{\prime\prime}(r) - \frac{h^\prime(r)}{r} \right)$ tends to $0$.
	As the tensor $\diff r \otimes \diff r$ has a bounded unit norm, their product strictly vanishes.
	This yields
	\[
		\lim_{y \rightarrow x_i} \operatorname{Hess}_y f_i = h^{\prime\prime}(0) \matheuvm{g} + 0 + 0 = h^{\prime\prime}(0) \matheuvm{g},
	\]
	which precisely matches the value at $y = x_i$.

	This continuous extension confirms that $\operatorname{Hess}_y f$ is continuous on the entirety of $\mathcal{U}$.
	Therefore, $\nabla_y f$ is indeed a $\mathcal{C}^1$ map on $\mathcal{U}$.
	As established in \Cref{lem:lipschitz_no_collision}, the joint inverse map $\mathcal{G}(\matheuvm{v}, y)$ is thus $\mathcal{C}^1$, and the operator norm of its continuous partial differential $\di_y \mathcal{G}$ achieves a finite maximum $L$ on the compact closure $\widebar{\mathcal{Q}}$.
	It follows that the sliced map $F_{\matheuvm{v}}$ is an $L$-Lipschitz continuous map on its respective domain, yielding the required uniform Lipschitz bound.
\end{proof}

\subsection{Decomposition of support and proof of the main theorem}
\label{sec:recursive_reduction}

We are now in a position to synthesize the decomposition strategy and the discrete approximation framework to prove the absolute continuity of the generalized Wasserstein barycenter.
We first establish the result under the assumption that all marginal measures are compactly supported.

\begin{thm}[Absolute continuity for compactly supported measures]
	\label{thm:absolute_continuity_compact}
	Let $(M, \matheuvm{g})$ be a complete Riemannian manifold.
	Suppose that $h:[0, \infty) \to[0, \infty)$ satisfies assumptions \ref{assumption:h_at_zero}-\ref{assumption:h_strict_convexity}.
	Given an integer $n \ge 2$,
	let $\mu_1, \ldots, \mu_n \in \mathcal{P}(M)$ be $n$ compactly supported probability measures
	with positive weights $\lambda_i > 0$ satisfying $\sum_{i=1}^n \lambda_i = 1$.
	If $\mu_1$ is absolutely continuous with respect to $\operatorname{Vol}$,
	then the unique generalized Wasserstein barycenter $\mu_\mathbb{P}$ of $\mathbb{P} := \sum_{i=1}^n \lambda_i \, \delta_{\mu_i}$ is also absolutely continuous in one of the following two different cases:
	\begin{enumerate}
		\item $h$ is $\mathcal{C}^2$-smooth on $[0, +\infty)$, i.e., $h^{\prime\prime}(0) := \lim_{t \downarrow 0} \frac{h^\prime(t)}{t} = \lim_{t \downarrow 0} h^{\prime\prime}(t)$;
		\item measures $\mu_2, \ldots, \mu_n$ are absolutely continuous.
	\end{enumerate}
\end{thm}

\begin{proof}
	The uniqueness of $\mu_\mathbb{P}$ follows from \Cref{thm:barycenter_uniqueness}.
	Let $\gamma \in \Pi(\mu_1, \ldots, \mu_n)$ be the unique multi-marginal optimal transport plan
	(\Cref{coro:multi_marginal_plan_structure})
	and let $B: M^n \rightarrow M$ be a measurable barycenter selection map.
	By \Cref{prop:barycenter_construction}, the barycenter is given by the push-forward $\mu_\mathbb{P} = B_{\#} \gamma$.

	Fix an arbitrarily chosen Borel set $N \in \mathcal{B}(M)$ such that $\operatorname{Vol}(N) = 0$.
	We shall show that $\mu_\mathbb{P}(N) = 0$ in both cases by partitioning the support of $\gamma$ into a countable collection of Borel sets based on the collision of the barycenter with the marginal points.

	Define $\widetilde{\gamma} := (\operatorname{Id}_{M^n}, B)_{\#} \gamma \in \mathcal{P}(M^{n+1})$.
	Since all measures $\mu_1, \ldots, \mu_n$ have compact support,
	\Cref{lem:barycenter_properties} implies that
	$\operatorname{supp}(\widetilde \gamma)$ is compact.
	As $\mu_\mathbb{P} = B_{\#} \gamma$, $\mu_\mathbb{P}(N) = \widetilde{\gamma}(M^n \times N)$.
	For each $1 \le i \le n$, define the collision set $C_i \subset M^{n+1}$ by
	\[
		C_i := \{ (\matheuvm{x}, z) \in M^{n+1} \mid z = x_i \}.
	\]
	The equality $z = x_i$ implies that,
	the mass of $M^n \times N$ restricted to each collision set is bounded by the mass of the corresponding marginal on $N$:
	\begin{equation}
		\label{eq:collision_bound}
		\widetilde{\gamma}(C_i \cap (M^n \times N)) \le \widetilde{\gamma}(\{ (\matheuvm{x}, z) \in M^{n+1} \mid x_i \in N \}) = \mu_i(N).
	\end{equation}
	Since $\mu_1$ is absolutely continuous with respect to the volume measure, $\mu_1(N) = 0$, which immediately implies $\widetilde{\gamma}(C_1 \cap (M^n \times N)) = 0$.

	Consider the \textbf{Case 1:} $h^{\prime\prime}(0) := \lim_{t \downarrow 0} \frac{h^\prime(t)}{t} = \lim_{t \downarrow 0} h^{\prime\prime}(t)$.
	We exhaust the regular region strictly away from the first collision set:
	\[
		\operatorname{supp}(\widetilde \gamma) \subset C_1 \cup \bigcup_{k=1}^\infty \mathcal{H}_k,
		\quad \text{ with }
		\mathcal{H}_k := \left\{ (\matheuvm{x}, z) \in \operatorname{supp}(\widetilde \gamma) \;\middle|\; d_{\matheuvm{g}}(z, x_1) > \frac{1}{k} \right\}.
	\]
	If $\widetilde{\gamma}(\mathcal{H}_k) = 0$,
	then trivially $\widetilde{\gamma}(\mathcal{H}_k \cap [M^n \times N]) = 0$.
	Otherwise, denote by
	$\widetilde \gamma^{(k)}$ the normalized probability measure of $\widetilde \gamma |_{\mathcal{H}_k}$,
	by $\gamma^{(k)}$ the push-forward of $\widetilde \gamma^{(k)}$
	via the projection $p_1: M^n \times M \rightarrow M^n$,
	and by $\mu_1^{(k)}, \ldots, \mu_n^{(k)}$ the $n$ marginals
	of $\gamma^{(k)}$ in this order.
	Define $\mathbb{P}^{(k)} := \sum_{i=1}^n \lambda_i \, \delta_{\mu_i^{(k)}}$.
	Note that $\gamma^{(k)}$ is the normalized restrictions of $\gamma$
	to the Borel set
	$(\operatorname{Id}_{M^n}, B)^{-1}(\mathcal{H}_k)$.
	It follows that $\gamma^{(k)}$ is a multi-marginal optimal transport plan thanks to the restriction property (c.f.\@ \cite[Theorem 5.19]{villani2009optimal}),
	and its first marginal measure
	$\mu_1^{(k)}$ is absolutely continuous.
	Then $\mu_{\mathbb{P}^{(k)}}: = B_{\#} \gamma^{(k)}$ is the unique
	barycenter of $\mathbb{P}^{(k)}$ and also the last marginal of $\widetilde {\gamma}^{(k)}$,
	which implies
	\begin{equation}
		\label{equa:reduced_barycenter_problem}
		\widetilde{\gamma}(\mathcal{H}_k \cap [M^n \times N]) = \widetilde{\gamma}(\mathcal{H}_k) \, \mu_{\mathbb{P}^{(k)}}(N).
	\end{equation}

	We now apply \Cref{prop:absolute_continuity_via_approximation} to the barycenter problem associated with $\mathbb{P}^{(k)}$ and the set $\Omega = \mathcal{H}_k$.
	Since $\mathcal{H}_k$ is a relatively open subset of $\operatorname{supp}(\widetilde \gamma)$
	(and thus of $\Gamma$),
	$\Gamma \setminus \mathcal{H}_k$ coincides with its closure,
	which implies the Condition 1 as $\widetilde {\gamma}^{(k)}$ assigns full mass to $\mathcal{H}_k$.
	Moreover,
	\Cref{lem:lipschitz_with_collision} applied to the compact set $K = \operatorname{supp}(\widetilde \gamma)$
	guarantees that Condition 2 is satisfied with a uniform Lipschitz constant $L_k$.
	Therefore, $\mu_{\mathbb{P}^{(k)}}$ is absolutely continuous
	and $\widetilde{\gamma}(\mathcal{H}_k \cap [M^n \times N]) = 0$ by \eqref{equa:reduced_barycenter_problem},
	which concludes the proof.

	Consider the \textbf{Case 2:} measures $\mu_2, \ldots, \mu_n$ are absolutely continuous.
	By \eqref{eq:collision_bound}, we have $\widetilde{\gamma}(C_i \cap (M^n \times N)) \le \mu_i(N) = 0$ for all $1 \le i \le n$.
	We exhaust the remaining region as follows:
	\[
		\operatorname{supp}(\widetilde \gamma) \subset \bigcup_{i=1}^n C_i \cup \bigcup_{k=1}^\infty \mathcal{H}_k^\prime,
		\quad \text{ with }
		\mathcal{H}_k^\prime := \left\{ (\matheuvm{x}, z) \in \operatorname{supp}(\widetilde \gamma) \;\middle|\; d_{\matheuvm{g}}(z, x_i) > \frac{1}{k}
		\text{ for } 1 \le i \le n \right\}.
	\]
	We repeat the identical restriction argument with $\Omega = \mathcal{H}_k^\prime$
	when $\widetilde \gamma(\mathcal{H}_k^\prime) \neq 0$.
	\Cref{lem:lipschitz_no_collision} ensures that Condition 2 of \Cref{prop:absolute_continuity_via_approximation} is satisfied.
	Hence, the barycenter of the restricted plan is absolutely continuous,
	implying $\widetilde{\gamma}(\mathcal{H}_k^\prime \cap [M^n \times N]) = 0$
	and thus the proof.
\end{proof}

We now state the main theorem in its full generality, removing the compactness assumption.
This is achieved via a decomposition of multi-marginal optimal transport plans.

\begin{coro}[Absolute continuity of generalized Wasserstein barycenters]
	\label{thm:main_theorem_general}
	Let $(M, \matheuvm{g})$ be a complete Riemannian manifold.
	Suppose that $h:[0, \infty) \to[0, \infty)$ satisfies assumptions \ref{assumption:h_at_zero}-\ref{assumption:h_strict_convexity}.
	Given an integer $n \ge 2$,
	let $\mu_1, \ldots, \mu_n \in \mathcal{W}_h(M)$ be $n$ probability measures
	with positive weights $\lambda_i > 0$ satisfying $\sum_{i=1}^n \lambda_i = 1$.
	If $\mu_1$ is absolutely continuous with respect to the volume measure $\operatorname{Vol}$,
	then the unique generalized Wasserstein barycenter $\mu_\mathbb{P}$ of $\mathbb{P} := \sum_{i=1}^n \lambda_i \, \delta_{\mu_i}$ is also absolutely continuous provided that one of the following two conditions holds:
	\begin{enumerate}
		\item $h^{\prime\prime}(0) := \lim_{t \downarrow 0} \frac{h^\prime(t)}{t} = \lim_{t \downarrow 0} h^{\prime\prime}(t)$;
		\item measures $\mu_2, \ldots, \mu_n$ are absolutely continuous.
	\end{enumerate}
\end{coro}

\begin{proof}
	By \Cref{prop:barycenter_construction}, the unique barycenter is given by the push-forward $\mu_\mathbb{P} = B_{\#} \gamma$, where $\gamma \in \Pi(\mu_1, \ldots, \mu_n)$ is a multi-marginal optimal transport plan
	and $B: M^n \rightarrow M$ is a measurable barycenter selection map.

	To decompose $\gamma$ into compactly supported measures,
	we explicitly construct a sequence of compact annuli with $\gamma$-negligible boundaries
	(cf.\@ \cite[Theorem 3.6]{ma2025thesis}).
	Fix a point $\matheuvm{x}_0 \in M^n$.
	Since the sum of uncountably many strictly positive real numbers is always infinite,
	there exists an unbounded and strictly increasing sequence $0= \alpha_0 < \alpha_1 < \alpha_2 < \cdots$
	such that $\gamma$ assigns zero mass to metric spheres
	centered at $\matheuvm{x}_0$ with radius $\alpha_i$ for all $i \ge 1$.
	Using annuli $K_j$ centered at $\matheuvm{x}_0$ with radii
	chosen from consecutive pairs of $\{\alpha_i\}_{i \ge 0}$,
	we can rewrite $\gamma$ as a weighted sum of at most countably many probability measures,
	\[
		\gamma = \gamma(K_1) \gamma_1 + \cdots + \gamma(K_j) \gamma_j + \cdots,
		\quad \text{ with } \gamma(K_j) > 0 \text{ and } \gamma_j = \frac{1}{\gamma(K_j)} \gamma |_{K_j}.
	\]
	Each measure $\gamma_k$ induces a barycenter problem for $n$ compactly supported measures,
	and $\mu_{\mathbb{P}_k}: = B_{\#} \gamma_k$ is the corresponding unique generalized Wasserstein barycenter
	thanks to the restriction property (c.f.\@ \cite[Theorem 5.19]{villani2009optimal}).
	According to \Cref{thm:absolute_continuity_compact}, $\mu_{\mathbb{P}_j}$
	is absolutely continuous,
	which implies that $\mu_{\mathbb{P}} = \gamma(K_1) \mu_{\mathbb{P}_1} + \cdots +
		\gamma(K_k) \mu_{\mathbb{P}_k} + \cdots$ is also absolutely continuous.
\end{proof}



\bibliography{bibliography.bib}
\end{document}